\documentclass[oneside,12pt]{amsart}
\usepackage[T2A]{fontenc}
\usepackage[cp1251]{inputenc}
\usepackage{amssymb}
\usepackage{amsxtra}
\usepackage{amsmath}
\usepackage{amstext,amsthm,amsfonts,amscd}

\usepackage[all]{xy}

\usepackage{hyperref}
\usepackage{tikz,graphicx}
\usetikzlibrary{matrix,arrows}

\usepackage[russian,english]{babel}

\DeclareMathOperator{\GL}{GL}

\DeclareMathOperator{\SL}{SL}

\DeclareMathOperator{\WW}{W}

\DeclareMathOperator{\Lie}{Lie}
\DeclareMathOperator{\Cent}{Cent}
\DeclareMathOperator{\rad}{rad}

\DeclareMathOperator{\Hom}{Hom}
\DeclareMathOperator{\Spec}{Spec}
\DeclareMathOperator{\Ga}{{\mathbf G}_a}
\DeclareMathOperator{\Gm}{{\mathbf G}_m}

\DeclareMathOperator{\Aut}{Aut}

\DeclareMathOperator{\Dyn}{Dyn}
\DeclareMathOperator{\Sym}{Sym}
\DeclareMathOperator{\Gal}{Gal}

\DeclareMathOperator{\der}{der}

\DeclareMathOperator{\rank}{rank}
\newcommand{\Aff}{\mathbb {A}}

\newcommand{\ad}{\mathrm{ad}}
\newcommand{\scl}{{sc}}

\newcommand{\qs}{\mathrm{qs}}

\DeclareMathOperator{\ZZ}{{\mathbb Z}}
\DeclareMathOperator{\QQ}{{\mathbb Q}}

\DeclareMathOperator{\et}{\text{\it \'et}}

\DeclareMathOperator{\La}{{\mathcal L}}

\newtheorem{lem}{Lemma}[section]
\newtheorem*{lem*}{Lemma}%[section]
\newtheorem*{thm*}{Theorem}
\newtheorem{thm}[lem]{Theorem}
\newtheorem{cor}[lem]{Corollary}

\newtheorem{nota}[lem]{Notation}

\theoremstyle{definition}{  \newtheorem{rem}[lem]{Remark}  }
\theoremstyle{definition}{  \newtheorem{exa}[lem]{Example} }
\theoremstyle{definition}{  \newtheorem{defn}[lem]{Definition} }

\let\l\left
\let\r\right
\let\semir\ltimes
\let\semil\rtimes

\newcommand{\st}{\scriptstyle}

\makeatletter

\@addtoreset{equation}{section}
\makeatother

\begin{document}

\selectlanguage{english}

\title{Non-stable $K_1$-functors of multiloop groups}

\author{A. Stavrova}
\thanks{
The author was supported at different stages of her work by the postdoctoral
grant 6.50.22.2014 ``Structure theory, representation theory and geometry of algebraic groups''
at St. Petersburg State University,
by the J.E. Marsden postdoctoral fellowship of the Fields Institute, and
by the RFBR grants 14-01-31515-mol\_a, 13-01-00429-a.}
\address{Department of Mathematics and Mechanics, St. Petersburg State University,
St. Petersburg, Russia}
\email{anastasia.stavrova@gmail.com}

\maketitle

\begin{abstract}
Let $k$ be a field of characteristic 0. Let $G$ be a reductive group over the ring of Laurent polynomials
$R=k[x_1^{\pm 1},...,x_n^{\pm 1}]$. Assume that $G$ contains a maximal $R$-torus, and
that every semisimple normal subgroup of $G$ contains a two-dimensional split torus $\mathbf{G}_m^2$.
We show that the natural map of non-stable $K_1$-functors, also called Whitehead groups,
$K_1^G(R)\to K_1^G\bigl( k((x_1))...((x_n)) \bigr)$ is injective, and an isomorphism if $G$ is semisimple.
As an application, we provide a way to compute the difference between the
full automorphism group of a Lie torus (in the sense of Yoshii-Neher) and the subgroup generated by
exponential automorphisms.

\end{abstract}

%\begin{center}
%\vspace{20pt}
%{\large A. Stavrova\footnote{
%The author was supported at different stages of her work by the postdoctoral
%grant 6.50.22.2014 ``Structure theory, representation theory and geometry of algebraic groups''
%at St. Petersburg State University,
%by the J.E. Marsden postdoctoral fellowship of the Fields Institute, and
%by the RFBR grants 14-01-31515-mol\_a, 13-01-00429-a.}}\\
%\vspace{10pt}
%Department of Mathematics and Mechanics\\St. Petersburg State University\\
%St. Petersburg, Russia\\
%anastasia.stavrova@gmail.com
%\end{center}

\section{Introduction}

Let $R$ be a commutative ring with 1, and let $G$ be a reductive group scheme over $R$ in the sense
of~\cite{SGA3}. We say that the group scheme $G$ is isotropic, if it contains a proper parabolic subgroup $P$,
or, equivalently, the automorphism group of $G$ contains a split 1-dimensional torus $\Gm$. Under this assumption
one can consider the following "large"{} subgroup of $G(R)$ generated by unipotent elements,
$E_P(R)=\l<U_P(R),U_{P^-}(R)\r>$
where $U_P$ and $U_{P^-}$ are the unipotent radicals of $P$ and any opposite parabolic subgroup $P^-$.
If $R$ is a field of characteristic $0$ and $G$ is the automorphism group of a
$\ZZ$-graded finite-dimensional Lie algebra $L$ over $R$,
then $E_P(R)$ can be visualized as the subgroup generated by $\exp(\ad(x))$, where $x$
runs over all elements of non-zero grading in $L$.

The set of (left) cosets
$$
G(R)/E_P(R)=K_1^{G,P}(R)
$$
is called the non-stable $K_1$-functor associated to $G$ and $P$~\cite{Stein,HV,W10}. When
$G(R)/E_P(R)$ is a group, it is also sometimes denoted by $W_P(R,G)$ and called
the Whitehead group of $G$~\cite{Abe,Gil,ChGP-Wh}.
Both names go back to Bass' founding paper~\cite{Bass}, where the case $G=\GL_n$ was considered.
We prefer the name "non-stable $K_1$-functor"{} over "Whitehead group"{},
since it suggests the existence of other non-stable $K$-functors. Indeed,
as a functor on the category of smooth algebras over a field, the non-stable $K_1$-functor
coincides with the first of
non-stable Karoubi--Villamayor $K$-functors in the sense of J.F. Jardine~\cite{J}, see~\cite{W10,St-poly}.

It is known that $K_1^{G,P}(R)$ is a group, and independent of the choice of $P$,
if $R$ is a semilocal ring, or if every semisimple normal subgroup of $G$ contains
$(\Gm)^2$ locally in Zariski topology on $\Spec R$~\cite{SGA3,Sus,PS} (see also~\S~\ref{sec:prel}).
In this case $P$ is omitted from the notation, i.e. we write $K_1^G$ instead of $K_1^{G,P}$.

If $R=k$ is a field, the group $K_1^G(k)=G(k)/G(k)^+$ was systematically studied since 1960s in relation to the Kneser-Tits
problem~\cite{Tits}, see the excellent survey~\cite{Gil}. In particular, if $G$ is simple and simply connected,
this group is known to be torsion, trivial in many cases (e.g. if
$G$ is $k$-rational), and abelian except possibly for some groups of type $E_8$~\cite{ChM-su,Gil}. The situation
for arbitrary commutative rings is much less clear. The next simplest case seems to be the one
of polynomial rings over a field $k$. It is known that if $G$ is "constant"{}, i.e. defined over $k$,
then $K_1^G(k[x])=K_1^G(k)$~\cite{M}, and if, moreover, every semisimple normal subgroup of $G$ contains
$(\Gm)^2$, then $$K_1^G(k[x_1,\ldots,x_n])=K_1^G(k)$$ for any $n\ge 1$~\cite{Sus,Abe,St-poly}. If
$G$ is simply connected, then also
$$
K_1^G\bigl(k[x_1^{\pm 1},\ldots,x_n^{\pm 1}]\bigr)=K_1^G(k)$$ for any $n\ge 1$~\cite{Sus,St-poly}.

The non-constant case,
where $G$ is defined over the polynomial ring itself, is not so well understood.
However, recently, significant progress was made by V. Chernousov, P. Gille, and A. Pianzola
in the case of a Laurent polynomial ring~\cite{ChGP-conj,ChGP-Wh}.
Their work is motivated by applications to the theory of infinite-dimensional Lie algebras,
namely, to classification and conjugacy problems for extended affine Lie algebras (EALAs), which are higher nullity
generalizations of affine Kac-Moody algebras~\cite{AABGP}. Any EALA can be reconstructed from its centerless
core, which is a Lie torus in the sense of~\cite{Y,N}, while the
Realization theorem~\cite[Theorem 3.3.1]{ABFP} implies that all Lie tori,
except for just one class called quantum tori, are Lie algebras of some isotropic adjoint simple
group schemes over $k[x_1^{\pm 1},\ldots,x_n^{\pm 1}]$ (see \S~\ref{sec:Lie} for details).

In~\cite{ChGP-Wh} V. Chernousov, P. Gille, and A. Pianzola showed
that $K_1^G(k[x^{\pm 1}])=1$ for a simply connected group $G$ defined
over $k[x^{\pm 1}]$, provided that $G$ contains $(\Gm)^2$, and either $G$ is quasi-split,
or $k$ is algebraically closed.
In~\cite{ChGP-conj} they obtained a general theorem relating groups of points
$G(k[x_1^{\pm 1},\ldots,x_n^{\pm 1}])$ and $G\bigl(k((x_1^{\pm 1}))\ldots((x_n^{\pm 1}))\bigr)$.
We state it here in a slightly simplified form.

\begin{thm}\label{thm:ChGP}\cite[Theorem 14.3]{ChGP-conj}
Let $k$ be a field of characteristic $0$, and set $R=k[x_1^{\pm 1},\ldots,x_n^{\pm 1}]$ and
$K=k((x_1))\ldots((x_n))$.
Let $G$ be a reductive group over $R$ having a maximal $R$-torus $T$.
%If $G_K$ contains a proper parabolic $K$-subgroup, let
%$P$ be a minimal parabolic $K$-subgroup of $G$, and set
%$E=E_P(K)$; otherwise set $E=1$.
Then there exists a subgroup $J$ of $G(K)$ such that
\begin{itemize}
\item $J$ has no non-trivial quotient groups of finite exponent;
%\item $J\cdot G(K)^+$ is a normal subgroup of $G(K)$
%\item $G(K)=\left<G(R),J,G(K)^+\right>$,
\item $G(K)=G(R)\cdot J\cdot G(K)^+$, where $G(K)^+$ stands for the normal subgroup of $G(K)$
generated by the $K$-points of all $K$-subgroups of $G_K$ isomorphic to $\Ga_{,K}$.
\end{itemize}
\end{thm}

Note that a reductive group $G$ over $R=k[x_1^{\pm 1},\ldots,x_n^{\pm 1}]$ always has a maximal $R$-torus,
if $n=1$~\cite[Propositions 5.9 and 5.10]{ChGP-line}, or if $k$ is algebraically closed and
$G$ is the adjoint group associated to a Lie torus~\cite[p. 532]{GiPi08}.
In general there are counterexamples~\cite[Remark 6.6]{GiPi-mem}.

In the setting of Theorem~\ref{thm:ChGP}, assume that $G$ is semisimple and isotropic.
Then the theorem implies that for any minimal parabolic $R$-subgroup $P$ of $G$ the natural map
\begin{equation}\label{eq:map}
K_1^{G,P}\bigl(k[x_1^{\pm 1},\ldots,x_n^{\pm 1}]\bigr)\to K_1^{G,P}\bigl(k((x_1))\ldots((x_n))\bigr)
\end{equation}
is surjective; see~\cite[Remark 14.4]{ChGP-conj} and Corollary~\ref{cor:dens-iso} in \S~\ref{sec:prel}.
This result is essential for the proof of main theorems in~\cite{ChGP-conj}, but, unfortunately, it does not
allow the computation of $K_1^G\bigl(k[x_1^{\pm 1},\ldots,x_n^{\pm 1}]\bigr)$.
V. Chernousov, P. Gille, and A. Pianzola then ask the following natural
question~\cite[p. 316]{ChGP-conj}: is the map~\eqref{eq:map} also injective, i.e. an isomorphism?
We answer this question positively in the following generality.

\begin{thm}\label{thm:main}
Let $k$ be a field of characteristic $0$. Let $G$ be a reductive group scheme
over $R=k[x_1^{\pm 1},\ldots,x_n^{\pm 1}]$ having a maximal $R$-torus $T$, and such that every semisimple
normal subgroup of $G$ contains $(\Gm_{,R})^2$.
Then the natural map
$$
K_1^G\bigl(k[x_1^{\pm 1},\ldots,x_n^{\pm 1}]\bigr)\to K_1^G\bigl(k((x_1))\ldots((x_n))\bigr)
$$
is injective. If $G$ is semisimple, this map is an isomorphism.
\end{thm}

Note that this theorem implies the above-mentioned results of~\cite{ChGP-Wh}, since for a quasi-split simply
connected group $G$, one has $K_1^G(K)=1$ for any field $K$ (see~\cite{Gil}).

Theorem~\ref{thm:main} is proved in~\S~\ref{ssec:mainproof} by combining the results of~\cite{PS,St-poly} on the structure of
isotropic groups over general commutative rings with a special "diagonal argument"{} trick inspired
by some unpublished work of Ivan Panin elaborating on~\cite[Prop. 7.1]{OPa}; see Lemma~\ref{lem:diag}. The
assumption that $G$ contains $(\Gm_{,R})^2$ and not just $\Gm_{,R}$ goes back to~\cite{Sus,PS}, the reason
being that $\SL_2(k[x,y])$ is not equal to its subgroup $E_2(k[x,y])$ generated by upper and lower unitriangular matrices,
and our methods fail. The actual statement of Theorem~\ref{thm:main}
for $K_1^G=\SL_2/E_2$ is trivially true if $n=1$ since $k[x^{\pm 1}]$ is Euclidean, and false
if $n\ge 3$ by~\cite{BaMo}; the case $n=2$ is not known at present, e.g.~\cite{Abr}.

As an immediate corollary of Theorem~\ref{thm:main}, we obtain the following result on Lie tori. Recall that a Lie torus is
a $\Delta\times\Lambda$-graded Lie algebra, where $\Delta$ is
an irreducible finite root system joined with $0$ and $\Lambda\cong\ZZ^n$,
satisfying certain axioms similar to the standard generators and relations axiomatics of complex simple Lie
algebras; see~Definition~\ref{defn:LT} in \S~\ref{sec:Lie}.

\begin{thm}\label{thm:Lie}
Let $k$ be an algebraically closed field of characteristic $0$,
$\Delta$ be a finite root system of rank $\ge 2$, and $\Lambda=\ZZ^n$, $n\ge 1$.
Let $\La$ be a centerless Lie $\Lambda$-torus of type $\Delta$ over $k$ that is finitely generated over its
centroid $R\cong k[x_1^{\pm 1},\ldots,x_n^{\pm 1}]$.
Let $G=\Aut_{R}(\La)^\circ$ be the connected component of the algebraic automorphism group of $\La$ as an $R$-Lie algebra,
and set
$$
E_{exp}(\La)=\l<\exp(\ad_x),\ x\in\La^\lambda_\alpha,\ (\alpha,\lambda)\in \Delta\times\Lambda,\ \alpha\neq 0\r>.
$$
Then there is an isomorphism of groups
\begin{equation}\label{eq:hom-Lietori}
G(R)/E_{exp}(\La)\cong K_1^G\bigl(k((x_1))\ldots((x_n))\bigr).
\end{equation}
%In particular, if $\La$ is quasi-split, then $G(R)=E_{exp}(\La)$.
\end{thm}

%Note that, provided~\cite[Theorem 14.3]{ChGP-conj}, the homomorphism~\eqref{eq:hom-Lietori} is also surjective,
%and hence an isomorphism.
% Is not trivial in many cases! The group here is adjoint, not simply connected.

Using the same methods as in Theorem~\ref{thm:main}, we also prove a similar statement on
${\mathcal R}$-equivalence class groups of Yu. Manin.
This application was suggested by Philippe Gille. The proof is given in \S~\ref{ssec:2fields}.

\begin{thm}\label{thm:2fields}
Let $k$ be a field of characteristic $0$. Let $G$ be a reductive group scheme
over $R=k[x_1^{\pm 1},\ldots,x_n^{\pm 1}]$ having a maximal $R$-torus $T$.
Then the natural map of ${\mathcal R}$-equivalence class groups
$$
G\bigl(k(x_1,\ldots,x_n)\bigr)/{\mathcal R}\to G\bigl(k((x_1))\ldots((x_n))\bigr)/{\mathcal R}
$$
is an isomorphism.
\end{thm}

\medskip
{\it Acknowledgements.}
This paper grew out of my conversations with Vladimir Chernousov and Philippe Gille, and I am very much
indebted to them for their numerous comments, suggestions, and corrections. I am also grateful to the
anonymous referee for providing a simplified proof of Lemma~\ref{lem:diag}.
%I am especially grateful to Philippe Gille for taking the trouble to read
%a preliminary version of this text.

\section{Preliminaries}\label{sec:prel}

\subsection{Elementary subgroups, non-stable $K_1$-functors, and $\mathcal{R}$-equivalence}

Let $A$ be a commutative ring. Let $G$ be an isotropic reductive group scheme over $A$, and
let $P$ be a parabolic subgroup of $G$ in the sense of~\cite{SGA3}.
Since the base $\Spec A$ is affine, the group $P$ has a Levi subgroup $L_P$~\cite[Exp.~XXVI Cor.~2.3]{SGA3}.
There is a unique parabolic subgroup $P^-$ in $G$ which is opposite to $P$ with respect to $L_P$,
that is $P^-\cap P=L_P$, cf.~\cite[Exp. XXVI Th. 4.3.2]{SGA3}.  We denote by $U_P$ and $U_{P^-}$ the unipotent
radicals of $P$ and $P^-$ respectively.

\begin{defn}
The \emph{elementary subgroup $E_P(A)$ corresponding to $P$} is the subgroup of $G(A)$
generated as an abstract group by $U_P(A)$ and $U_{P^-}(A)$.
\end{defn}

Note that if $L'_P$ is another Levi subgroup of $P$,
then $L'_P$ and $L_P$ are conjugate by an element $u\in U_P(A)$~\cite[Exp. XXVI Cor. 1.8]{SGA3}, hence
$E_P(A)$ does not depend on the choice of a Levi subgroup or of an opposite subgroup
$P^-$, respectively. We suppress the particular choice of $L_P$ or $P^-$ in this context.

\begin{defn}
A parabolic subgroup $P$ in $G$ is called
\emph{strictly proper}, if it intersects properly every normal semisimple subgroup of $G$.
\end{defn}

%Equivalently, $P$ is strictly proper, if for every maximal ideal $m$ in $A$ the image of $P_{A_m}$ in
%$G_i$ under the projection map is a proper subgroup in $G_i$, where $G^{\ad}_{A_m}=\prod_i G_i$
%is the decomposition of the adjoint semisimple group $G_{A_m}^{\ad}$ into a product of absolutely almost simple groups.

%every normal semisimple subgroup of $G_{A_m}$ contains $(\Gm)^2$, or, equivalently,
%all irreducible components of the
%root system of $G_{A_m}$ relative to a maximal split subtorus in the sense of~\cite[Exp. XXVI, \S 7]{SGA} are of rank $\ge 2$.

The following theorem combines several results of~\cite{PS} and~\cite[Exp. XXVI, \S 5]{SGA3}.

\begin{thm}\label{thm:EE}\cite[Theorem 2.1]{St-poly}
Let $G$ be a reductive group scheme over a commutative ring $A$, and let $R$ be a commutative $A$-algebra.

(i) Assume that $A$ is a semilocal ring. Then the subgroup $E_P(R)$ of $G(R)$ is the same for any
minimal parabolic $A$-subgroup $P$ of $G$. If, moreover, $G$ contains a strictly proper parabolic $A$-subgroup,
the subgroup $E_P(R)$ is the same for any strictly proper parabolic $A$-subgroup $P$.

(ii) If $A$ is not necessarily semilocal, but for every maximal ideal $m$ in $A$, every normal semisimple
subgroup of $G_{A_m}$ contains $(\Gm_{,A_m})^2$, then the subgroup $E_P(R)$ of $G(R)=G_R(R)$ is the same for any
strictly proper parabolic $R$-subgroup $P$ of $G_R$.

In both these cases $E_P(A)$ is normal in $G(A)$.
\end{thm}

\begin{defn}
Under the assumptions of Theorem~\ref{thm:EE} (i) or (ii), we call $E_P(R)$ \emph{the elementary
subgroup} of $G(R)$ and denote it by $E(R)$.
\end{defn}

\begin{defn}
The functor $K_1^{G,P}(R)=G(R)/E_P(R)$ on the category of commutative $A$-algebras $R$
is called the \emph{non-stable $K_1$-functor, or the Whitehead group} associated to $G$ and $P$.
Under the assumptions of Theorem~\ref{thm:EE} (i) or (ii), we write $K_1^G$ instead of $K_1^{G,P}$.
%$$
%R\mapsto K_1^G(R)=G(R)/E(R),
%$$ on the category
%of commutative unital $A$-algebras $R$.
\end{defn}

Note that the normality of the elementary subgroup implies that $K_1^G$ is a group-valued functor.
%\begin{rem}
%If $A$ is a perfect field, then on the category of regular $A$-algebras $K_1^G$ coincides with
%the first non-stable Karoubi-Villamayor $K$-functor $KV_1^G$~\cite[Lemma 3.3 and Theorem 1.3]{St-poly},
%and shares many other properties of the usual algebraic $K_1$-functor, see~\cite{St-poly,W10}.
%These results actually extend to any field $A$ by means of~\cite{PaStV} (in preparation).
%This is why we prefer the notation $K_1^G$ to $W(-,G)$ of~\cite{ChGP-Wh}.
%\end{rem}

Non-stable $K_1$ functors are closely related to ${\mathcal R}$-equivalence class groups introduced by Yu. Manin
in~\cite{Man}.

\begin{defn}
Let $X$ be an algebraic variety over a field $k$.
Denote by $k[t]_{(t),(t-1)}$ the semilocal ring
of the affine line $\Aff^1_k$ over $k$ at the points $0$ and $1$. Two points $x_0,x_1\in X(k)$ are
called \emph{directly ${\mathcal R}$-equivalent}, if there is $x(t)\in X(k[x]_{(x),(x-1)})$ such that $x(0)=x_0$ and $x(1)=x_1$.
The \emph{${\mathcal R}$-equivalence relation} on $X(k)$ is the equivalence relation generated by direct ${\mathcal R}$-equivalence.
The \emph{${\mathcal R}$-equivalence class group $G(k)/{\mathcal R}$} of an algebraic $k$-group $G$ is the quotient of $G(k)$
by the ${\mathcal R}$-equivalence class of the neutral element $1\in G(k)$.
\end{defn}

It is easy to see that the ${\mathcal R}$-equivalence class of the neutral element $1\in G(k)$ is a normal subgroup
of $G(k)$, so $G(k)/{\mathcal R}$ is indeed a group. Apart from that, if $G$ has a proper parabolic subgroup $P$
over $k$, then all elements of $E_P(k)$ are ${\mathcal R}$-equivalent to $1$, so $K_1^{G,P}(k)$ surjects
onto $G(k)/{\mathcal R}$. If $G$ semisimple and simply connected, and $P$ is strictly proper, then
$K_1^{G,P}(k)=K_1^G(k)\cong G(k)/{\mathcal R}$ by~\cite[Th\'eor\`eme 7.2]{Gil}.

In the present paper we are mainly interested in values of $K_1^G$ on Laurent polynomial
rings over a field. We will use the following result.

\begin{thm*}\cite[Corollary 6.2]{St-poly}
Let $G$ be a simply connected semisimple algebraic group over a field $k$, such that every
semisimple normal subgroup of
$G$ contains $(\Gm_{,k})^2$. For any $m,n\ge 0$, there are natural isomorphisms
$$
K_1^G(k)\cong K_1^G\l(k[Y_1,\ldots,Y_m,X_1,X_1^{-1},\ldots,X_n,X_n^{-1}]\r).
%\cong K_1^G\bigl(k(Y_1,\ldots,Y_m,X_1,\ldots,X_n)\bigr)
$$
\end{thm*}

We will also use the following lemma, that was established in~\cite[Corollary 5.7]{Sus} for $G=\GL_n$, and
in~\cite[Prop. 3.3]{Abe} for most Chevalley groups; for isotropic groups
it was proved in~\cite[Lemma 6.1]{St-poly}, although the statement was slightly weaker than the present one.
The idea goes back to~\cite{Q}.

\begin{lem}\label{lem:f}
Let $A$ be a commutative ring, and let $G$ be a reductive group scheme over $A$, such that every semisimple normal subgroup
of $G$ is isotropic. Assume moreover that for any maximal ideal $m\subseteq A$,
 every semisimple normal subgroup of $G_{A_m}$ contains $(\Gm_{,A_m})^2$.
Then for any monic polynomial $f\in A[t]$ the natural homomorphism
$$
K_1^G(A[t])\to K_1^G(A[t]_f)
$$
is injective.
\end{lem}
\begin{proof}
The proof goes exactly as in~\cite[Corollary 5.7]{Sus} using Theorem~1.1 of~\cite{St-poly} in place of
Theorem 5.1 of~\cite{Sus}, and Lemma 2.3 of~\cite{St-poly} in place of Lemma~3.7 of~\cite{Sus}.
\end{proof}

\subsection{Torus actions on reductive groups}

Let $R$ be a commutative ring with 1, and let $S=(\Gm_{,R})^N=\Spec(R[x_1^{\pm 1},\ldots,x_N^{\pm 1}])$
be a split $N$-dimensional torus over $R$. Recall that the character group
$X^*(S)=\Hom_R(S,\Gm_{,R})$ of $S$ is canonically isomorphic to $\ZZ^N$.
If $S$ acts $R$-linearly on an $R$-module $V$, this module has a natural $\ZZ^N$-grading
$$
V=\bigoplus_{\lambda\in X^*(S)}V_\lambda,
$$
where
$$
V_\lambda=\{v\in V\ |\ s\cdot v=\lambda(s)v\ \mbox{for any}\ s\in S(R)\}.
$$
Conversely, any $\ZZ^N$-graded $R$-module $V$ can be provided with an $S$-action by the same rule.

Let $G$ be a reductive group scheme over $R$ in the sense of~\cite{SGA3}. Assume that $S$ acts on $G$
by $R$-group automorphisms. The associated Lie algebra functor $\Lie(G)$ then acquires
a $\ZZ^N$-grading compatible with the Lie algebra structure,
$$
\Lie(G)=\bigoplus_{\lambda\in X^*(S)}\Lie(G)_\lambda.
$$

We will use the following version of~\cite[Exp. XXVI Prop. 6.1]{SGA3}.

\begin{lem}\label{lem:T-P}
%Let $G$ be a reductive group over a commutative ring $R$, and let $S$ be a split torus of rank $n$
%acting on $G$ by group automorphisms.
Let $L=\Cent_G(S)$ be the subscheme of $G$ fixed by $S$. Let
$\Psi\subseteq X^*(S)$ be an $R$-subsheaf of sets closed under addition of characters.

(i) If $0\in\Psi$, then there exists
a unique smooth connected closed subgroup $U_\Psi$ of $G$ containing $L$ and satisfying
\begin{equation}\label{eq:LieUPsi}
\Lie(U_\Psi)=\bigoplus_{\lambda\in\Psi}\Lie(G)_\lambda.
\end{equation}
Moreover, if $\Psi=\{0\}$, then $U_\Psi=L$; if $\Psi=-\Psi$, then $U_\Psi$ is reductive; if $\Psi\cup(-\Psi)=X^*(S)$,
then $U_\Psi$ and $U_{-\Psi}$ are two opposite parabolic subgroups of $G$ with the common Levi subgroup
$U_{\Psi\cap(-\Psi)}$.

(ii) If $0\not\in\Psi$, then there exists a unique smooth connected unipotent closed subgroup $U_\Psi$ of $G$
normalized by $L$ and satisfying~\eqref{eq:LieUPsi}.
%$$
%\Lie(U_\Psi)=\bigoplus_{\lambda\in\Psi}\Lie(G)_\lambda.
%$$
\end{lem}
\begin{proof}
The statement immediately follows by faithfully flat descent from the standard facts about the subgroups of
split reductive groups proved in~\cite[Exp. XXII]{SGA3}; see the proof of~\cite[Exp. XXVI Prop. 6.1]{SGA3}.
\end{proof}

\begin{defn}
The sheaf of sets
$$
\Phi=\Phi(S,G)=\{\lambda\in X^*(S)\setminus\{0\}\ |\ \Lie(G)_\lambda\neq 0\}
$$
is called the \emph{system of relative roots of $G$ with respect to $S$}.
\end{defn}

\begin{rem}\label{rem:S-P}
Choosing a total ordering on the $\QQ$-space $\QQ\otimes_{\ZZ} X^*(S)\cong\QQ^n$, one defines the subsets
of positive and negative relative roots $\Phi^+$ and $\Phi^-$, so that $\Phi$ is a disjoint
union of $\Phi^+$, $\Phi^-$, and $\{0\}$. By Lemma~\ref{lem:T-P} the closed subgroups
$$
U_{\Phi^+\cup\{0\}}=P,\qquad U_{\Phi^-\cup\{0\}}=P^-
$$
are two opposite parabolic subgroups of $G$ with the common Levi subgroup $\Cent_G(S)$.
Thus, if a reductive group $G$ over $R$ admits a non-trivial action of a split torus,
then it has a proper parabolic subgroup. The converse is true Zariski-locally, see~Lemma~\ref{lem:relroots} below.
\end{rem}

\subsection{Loop reductive groups and maximal tori}\label{ssec:loop}

Let $k$ be a field of characteristic $0$. We fix once and for
all an algebraic closure $\bar k$ of $k$ and a compatible set of primitive $m$-th roots of unity
$\xi_m\in \bar k$, $m\ge 1$.

P. Gille and A. Pianzola~\cite[Ch. 2, 2.3]{GiPi-mem} compute the \'etale
(or algebraic) fundamental group of the $k$-scheme
$$
X=\Spec k[x_1^{\pm 1},\ldots, x_n^{\pm 1}]
$$ at the
natural geometric point $e:\Spec\bar k\to X$ induced by the evaluation $x_1=x_2=\ldots=x_n=1$. Namely,
let $k_\lambda$, $\lambda\in\Lambda$ be the set of finite Galois extensions of $k$ contained in $\bar k$.
Let $I$ be the subset of $\Lambda\times\ZZ_{>0}$ consisting of all pairs $(\lambda,m)$ such that
$\xi_m\in k_\lambda$. The set I is directed by the relation $(\lambda,m)\le (\mu,k)$ if and only if
$k_\lambda\subseteq k_\mu$ and $m|k$. Consider
$$
X_{\lambda,m}=\Spec k_\lambda[x_1^{\pm\frac1m},\ldots,x_n^{\pm\frac1m}]
$$
as a scheme over $X$ via the natural inclusion of rings. Then $X_{\lambda,m}\to X$ is a Galois cover
with the Galois group
$$
\Gamma_{\lambda,m}=(\ZZ/m\ZZ)^n\semil\Gal(k_\lambda/k),
$$
where $\Gal(k_\lambda/k)$ acts on $k_\lambda[x_1^{\pm\frac1m},\ldots,x_n^{\pm\frac1m}]$ via its canonical action
on $k_\lambda$, and each $(\bar k_1,\ldots,\bar k_n)\in(\ZZ/m\ZZ)^n$ sends $x_i^{1/m}$ to $\xi_m^{k_i}x_i^{1/m}$,
$1\le i\le n$. The semi-direct product structure on $\Gamma_{\lambda,m}$ is induced by
 the natural action of $\Gal(k_\lambda/k)$ on $\mu_m(k_\lambda)\cong \ZZ/m\ZZ$.
We have
\begin{equation}\label{eq:pi1X}
\pi_1(X,e)=\varprojlim\limits_{(\lambda,m)\in I}\Gamma_{\lambda,m}=\hat\ZZ(1)^n\semil\Gal(k),
\end{equation}
where $\hat\ZZ(1)$ denotes the profinite group $\varprojlim\limits_{m}\mu_m(\bar k)$ equipped with the
natural action of the absolute Galois group $\Gal(k)=\Gal(\bar k/k)$.

%For any field extension $l/k$, we denote
%$$
%l[x_1^{\pm\frac1\infty},\ldots,x_n^{\pm\frac1\infty}]=\varinjlim\limits_m l[x_1^{\pm\frac1m},\ldots,x_n^{\pm\frac1m}],
%$$
%where the limit is taken with respect to inclusion.

For any reductive group scheme $G$ over $X$, we denote by $G_0$ the split, or Chevalley---Demazure
reductive group in the sense of~\cite{SGA3} of the same type as $G$. The group $G$ is a twisted form of $G_0$, corresponding
to a cocycle class $\xi$ in the \'etale cohomology set $H^1_{\et}(X,\Aut(G_0))$.

\begin{defn}\label{def:loop}\cite[Definition 3.4]{GiPi-mem}
The group scheme $G$ is called \emph{loop reductive}, if the cocycle $\xi$ is in the image of the natural map
$$
H^1\bigl(\pi_1(X,e),\Aut(G_0)(\bar k)\bigr)\to H^1_{\et}\bigl(X,\Aut(G_0)\bigr).
$$
Here $H^1\bigl(\pi_1(X,e),\Aut(G_0)(\bar k)\bigr)$ stands for the non-abelian cohomology set
in the sense of Serre~\cite{Se}.
The group $\pi_1(X,e)$ acts continuously on $\Aut(G_0)(\bar k)$ via the natural homomorphism
$\pi_1(X,e)\to \Gal(\bar k/k)$.
\end{defn}

We will use the following result.

\begin{thm*}\cite[Corollary 6.3]{GiPi-mem}
A reductive group scheme over $X$
is loop reductive if and only if $G$ has a maximal torus over $X$.
\end{thm*}

The definition of a maximal torus is as follows.

\begin{defn}\cite[Exp. XII D\'ef. 3.1]{SGA3}
Let $G$ be a group scheme of finite type over a scheme $S$, and let $T$ be a $S$-torus
which is an $S$-subgroup scheme of $G$.
Then $T$ is a \emph{maximal torus of $G$ over $S$}, if $T_{\overline{k(s)}}$ is a maximal torus of
$G_{\overline{k(s)}}$ for all $s\in S$.
\end{defn}

\subsection{Surjectivity theorem of Chernousov---Gille---Pianzola}\label{ssec:surj}

In this section we discuss Theorem~\ref{thm:ChGP} stated in the introduction and its implications.

\begin{proof}[Proof of Theorem~\ref{thm:ChGP}.]
In the original statement of~\cite[Theorem 14.3]{ChGP-conj}, one considers
a linear algebraic $k$-group $H$ whose connected component of identity $H^\circ$ is reductive, and
a cocycle $\eta\in H^1(\pi_1(R,e),H(\bar k))$. Let  $\mathfrak{H}$ be the $R$-group scheme which is
the $\eta$-twisted form of $H_R$. Then there is a minimal parabolic (not necessarily proper) $R$-subgroup
scheme $\mathfrak{P}$ of
$\mathfrak{H}^\circ$, a Levi subgroup $\mathfrak{L}$ of $\mathfrak{P}$ which is a loop reductive group scheme,
and a normal subgroup $J$ of $\mathfrak{L}(K)$ such that
\begin{equation}\label{eq:ChGM-gen}
\mathfrak{H}(K)=\left<\mathfrak{H}(R),J,\mathfrak{H}(K)^+\right>,
\end{equation}
and $J$ is isomorphic to a quotient of a group admitting a composition
series whose quotients are pro-solvable groups in $k$-vector spaces.

Clearly, such a group $J$
has no non-trivial quotients of finite exponent. We also claim that in the above setting,
\begin{equation}\label{eq:Hcirc}
\mathfrak{H}^\circ(K)=\mathfrak{H}^\circ(R)\cdot J\cdot\mathfrak{H}^\circ(K)^+,
\end{equation}
where $J\cdot\mathfrak{H}^\circ(K)^+$ is normal in $\mathfrak{H}^\circ(K)$.
Since
$\mathfrak{H}^\circ$ is a loop reductive group, by~\cite[Corollary 7.4]{GiPi-mem} the parabolic subgroup
$\mathfrak{P}_K$ of $\mathfrak{H}^\circ_K$ is also minimal.
Then, since $K$ has characteristic $0$, one has $\mathfrak{H}^\circ(K)^+=E_{\mathfrak{P}}(K)$ by Theorem~\ref{thm:EE} (i)
and~\cite[Proposition 6.2]{BoTi73}.
By~\cite[Proposition 6.11]{BoTi73} one has
$$
\mathfrak{H}^\circ(K)=\mathfrak{L}(K)E_{\mathfrak{P}}(K).
$$
This implies that $J\cdot \mathfrak{H}^\circ(K)^+$ is normal in $\mathfrak{H}^\circ(K)$.
It remains to note that $\mathfrak{H}(K)^+=\mathfrak{H}^\circ(K)^+$, since $\Ga_{,K}$ is connected.
 Now~\eqref{eq:ChGM-gen}
and the equality $\mathfrak{H}^\circ(R)=\mathfrak{H}(R)\cap \mathfrak{H}^\circ(K)$ imply~\eqref{eq:Hcirc}.

We proceed to show how the above facts imply the claim of our theorem.

%In our simplified version of the theorem, we only assume that the $R$-group $G$ has a maximal $R$-torus, i.e. $G$ is
%loop reductive, but not necessarily given by a cocycle with values in an $R$-form of $G$ defined over $k$.

%If $G^\circ$ has a proper parabolic $R$-subgroup,
%then $P$ is defined as a minimal parabolic subgroup of $G^\circ$ having a loop reductive Levi subgroup $L$.
\emph{Case 1: $G$ is a torus.}  The proof of the theorem of Chernousov, Gille, and Pianzola for the case where
$\mathfrak{H}=\mathfrak{H}^\circ=G$ is an $R$-torus does not use the
assumption that $\mathfrak{H}$ is given by a cocycle with values in
$H(\bar k)$~\cite[Proof of Theorem 14.3, Case 1, p. 314]{ChGP-conj}. Therefore,~\eqref{eq:Hcirc} implies
that our theorem holds for $G$.

\emph{Case 2: $G$ is adjoint.} Assume that $G$ is a loop semisimple group of adjoint type over $R$.
Then $G=\Aut(G)^\circ$,
%There is a short exact sequence
%of $R$-group schemes
%$$
%1\to G\to\Aut(G)\to \Out(G)\to 1,
%$$
where $\Aut(G)$ is the $R$-group scheme of automorphisms of $G$.
%, and $\Out(G)$ is a finite group scheme over $R$.
%In particular, $G=\Aut(G)^\circ$.
Since $G$ is loop reductive, the group $\Aut(G)=\mathfrak{H}$ satisfies the
conditions of~\cite[Theorem 14.3]{ChGP-conj}.
Then~\eqref{eq:Hcirc} shows that the claim of our theorem holds for $G$. Note that in this case
$J\cdot G(K)^+$ is normal in $G(K)$.

\emph{Case 3: $G$ is semisimple.} Now assume that $G$ is an arbitrary loop semisimple group scheme over $R$.
Then there is a short exact sequence of $R$-group schemes
\begin{equation}\label{eq:GGad}
1\to\Cent(G)\to G\xrightarrow{p} G^{\ad}\to 1,
\end{equation}
where $G^{\ad}$ is an adjoint semisimple group, and $\Cent(G)$ is a finite group scheme of multiplicative type.
Since $G^{\ad}$ has a maximal $R$-torus if and only if $G$ does, $G^{\ad}$ is a loop semisimple group.
By the previous case
$$
G^{\ad}(K)=G^{\ad}(R)\cdot J\cdot G^{\ad}(K)^+,
$$
where $J$ has no non-trivial quotients of finite exponent. By~\cite[Corollaire 6.3]{BoTi73} we have
$p(G(K)^+)=G^{\ad}(K)^+$. Since $H^1_{\et}(K,\Cent(G))$ is a group of finite exponent, considering
the "long"{} exact sequence of \'etale cohomology associated to~\eqref{eq:GGad}, we conclude that
$J\subseteq p(G(K))$. Set $I=p^{-1}(J)\subseteq G(K)$.
Then, clearly,
$$
G(K)=p^{-1}(G^{\ad}(R))\cdot I\cdot G(K)^+.
$$
Since
$H^i_{\et}(R,\Cent(G))=H^i_{\et}(K,\Cent(G))$ for all $i\ge 0$ by~\cite[Prop. 3.4 (2)]{GiPi08}, the "long"{}
exact sequence also implies that
$$
p^{-1}(G^{\ad}(R))=\Cent(G)(K)\cdot G(R)=\Cent(G)(R)\cdot G(R)=G(R).
$$
%Consequently,
%$$
%G(K)=G(R)\cdot I\cdot G(K)^+.
%$$

Assume that $I$ has a proper normal subgroup $I'$ such that $I/I'$ has finite exponent. Since $J$
has no non-trivial quotients of finite exponent, we have $I'/\Cent(G)(K)\cap I'=J$,
and hence $I=\Cent(G)(K)\cdot I'$.
% In particular, $I/I'\le \Cent(G)(K)/\Cent(G)(K)\cap I'$.
Since $\Cent(G)(K)$ is finite, we can find a minimal subgroup $I'\le I$ such that
$I'$ is normal in $I$ and $I/I'$ has finite exponent. One readily sees that such $I'$
has no non-trivial quotients of finite exponent.
Since $\Cent(G)(K)=\Cent(G)(R)$, we have
$$
G(K)=p^{-1}(G^{\ad}(R))\cdot I\cdot G(K)^+=G(R)\cdot\Cent(G)(K)\cdot I'\cdot G(K)^+=G(R)\cdot I'\cdot G(K)^+,
$$
which proves the claim of the theorem for $G$.

%By Cor. 4.1.7 Exp. XXII
%Cent lies in T.

\emph{Case 4: $G$ is reductive.} Let $G$ be an arbitrary loop reductive group scheme over $R$. Let $\der(G)$ be the derived subgroup scheme of $G$ and let $\rad(G)$
be the radical torus of $G$ in the sense of~\cite{SGA3}.
By~\cite[Exp. XXII, Prop. 6.2.4]{SGA3}
there is a short exact sequence of $R$-group schemes
$$
1\to C\to \rad(G)\times\der(G)\xrightarrow{f} G\to 1,
$$
where $C$ is a finite group scheme of multiplicative type which is central in $\rad(G)\times\der(G)$.
Arguing exactly as in~\cite[Proof of Theorem 14.3, Case 2, pp. 314--315]{ChGP-conj} (except
that the reference to Theorem 11.1 ibid. should be replaced by that to Theorem 14.1 ibid.),
one concludes that
$$
G(K)=G(R)\cdot f(\der(G)(K)\times\rad(G)(K)).
$$
Note that $\der(G)$ is a loop semisimple group scheme, since it has a maximal $R$-torus once $G$ does.
Then $\rad(G)$ and $\der(G)$ are subject to the previous cases of the theorem, and one readily
deduces the claim for $G$.
\end{proof}

\begin{cor}\cite[Remark 14.4]{ChGP-conj}\label{cor:dens-iso}
Let $k,R,K,G$ be as in Theorem~\ref{thm:ChGP}. Assume in addition that $G$ is semisimple.
Then for any minimal parabolic $R$-subgroup $P$ of $G$ the map
$$
K_1^{G,P}\bigl(k[x_1^{\pm 1},\ldots,x_n^{\pm 1}]\bigr)\to K_1^{G,P}\bigl(k((x_1))\ldots((x_n))\bigr)
$$
is surjective.
\end{cor}
\begin{proof}
Since $G$ is a loop reductive group, by~\cite[Corollary 7.4]{GiPi-mem} any minimal parabolic subgroup
$P$ of $G$ remains a minimal parabolic subgroup in $G_K$.
Then, since $K$ has characteristic $0$, one has $G(K)^+=E_P(K)$ by~\cite[Proposition 6.2]{BoTi73}.
It was observed in~\cite[Remark 14.4]{ChGP-conj} that if $G$ is simply connected, then
the surjectivity of the  map in question follows from Theorem~\ref{thm:ChGP}, since the group
$K_1^{G,P}(K)$ has finite exponent by~\cite[Remarque 7.6]{Gil}. We claim that $K_1^{G,P}(K)$ has finite
exponent whenever $G$ is semisimple. Indeed, there is a short exact sequence
\begin{equation}\label{eq:scl}
1\to C\to G^{\scl}\to G\to 1,
\end{equation}
where $C$ is a finite group scheme of multiplicative type, contained in the center of $(G)^{\scl}$. Let
$P^{\scl}\subseteq (G)^{\scl}$ be the parabolic subgroup which is the preimage of $P$.
The "long" exact sequence of \'etale cohomology corresponding to~\eqref{eq:scl} readily shows that $K_1^{G,P}(K)$ has finite
exponent once $K_1^{G^{\scl},P^{\scl}}(K)$ does,
since $H^1_{\et}(K,C)$ is an abelian torsion group.
\end{proof}

We also obtain the following immediate corollary on ${\mathcal R}$-equivalence class groups. Note that the group
$G$ is not required to have a maximal torus over $k[x_1^{\pm},\ldots,x_n^{\pm 1}]$.

\begin{cor}\label{cor:2fields-surj}
Let $k$ be a field of characteristic $0$, and
let $G$ be a reductive group over $k[x_1^{\pm 1},\ldots,x_n^{\pm 1}]$. Set $F=k(x_1,\ldots,x_n)$.

(i) The natural map of ${\mathcal R}$-equivalence class groups
$$
G\bigl(k(x_1,\ldots,x_n)\bigr)/{\mathcal R}\to G\bigl(k((x_1))\ldots((x_n))\bigr)/{\mathcal R}
$$
is surjective.

(ii) If $G$ is semisimple, then for every strictly proper parabolic subgroup $P$ of $G_F$ the natural
map
$$
K_1^{G_F,P}\bigl(k(x_1,\ldots,x_n)\bigr)\to K_1^{G_F,P}\bigl(k((x_1))\ldots((x_n))\bigr)
$$
is surjective.
\end{cor}
\begin{proof}
The proof goes by induction on $n$ starting with $n=1$. Set $A=k[x^{\pm 1}]$, $F=k(x)$ and $K=k((x))$.
By~\cite[Propositions 5.9 and 5.10]{ChGP-line} every
semisimple group scheme $G$ over $A$ is loop reductive, i.e. contains a maximal $A$-torus.

By the definition of ${\mathcal R}$-equivalence the subgroup
$G(K)^+$ is contained in the ${\mathcal R}$-equivalence class of the neutral element. By~\cite[\S 17.1, Corollary 2]{Vos} the group $G\bigl(K\bigr)/{\mathcal R}$
has finite exponent. Therefore, by Theorem~\ref{thm:ChGP} the natural map $G(A)\to G(K)/{\mathcal R}$ is surjective.
Since this map factors through the map $G(F)/{\mathcal R}\to G(K)/{\mathcal R}$, the latter map is surjective.

%If $n=1$ and every semisimple normal subgroup of $G$ is isotropic, by~\cite[Lemma 4.3]{ChGP-Wh} we have that $G(k[x^{\pm 1}])$ surjects onto
%$K_1^G(k((x)))$; consequently, $K_1^G(k(x^{\pm 1}))\to K_1^G(k((x)))$ is also surjective.

Now consider the non-stable $K_1$-functors.
By~\cite[Corollary 7.4]{GiPi-mem} minimal parabolic subgroups of $G$, $G_{F}$ and $G_K$ are of the same type.
%Assume first that $G$ is anisotropic over $A$. Then by~\cite[Proposition 6.2]{BoTi73} we have
%$G\bigl(k((x))\bigr))
%$
Then, since $G_F$ contains a strictly proper parabolic $F$-subgroup $P$, we conclude that any minimal
parabolic $A$-subgroup $Q$ of $G$ is strictly proper. Moreover, $Q_F$ and $Q_K$ are minimal parabolic
subgroups of $G_F$ and $G_K$ respectively.
By Theorem~\ref{thm:EE} we have $K_1^{G,Q}(F)=K_1^{G_F,P}(F)$ and
$$
K_1^{G,Q}(K)=K_1^{G_K,P_K}(K)=K_1^{G_F,P}(K).
$$
By Corollary~\ref{cor:dens-iso} the natural map
$K_1^{G,Q}(A)\to K_1^{G,Q}(K)$
is surjective. Since this map factors through the map
$K_1^{G,Q}(F)\to K_1^{G,Q}(K)$, the latter map is also surjective. Therefore, the map
$$
K_1^{G_F,P}(F)\to K_1^{G_F,P}(K)
$$
is surjective.

Assume that $n>1$. Let $\mathcal{F}(-)$ denote any of the functors $K_1^{G_F,P}(-)$ and $G(-)/{\mathcal R}$ on the category
of field extensions of $k(x_1,\ldots,x_n)$. By the case $n=1$ the map
$$
\mathcal{F}\bigl(k((x_1))\ldots k((x_{n-1}))(x_n)\bigr)\to \mathcal{F}\bigl(k((x_1))\ldots((x_n))\bigr)
$$
is surjective. Since this map factors through the map
$$
\mathcal{F}\bigl(k(x_n)((x_1))\ldots ((x_{n-1}))\bigr)\to \mathcal{F}\bigl(k((x_1))\ldots((x_n))\bigr),
$$
the latter map is also surjective. By the induction hypothesis the map
$$
\mathcal{F}\bigl(k(x_1,\ldots, x_n)\bigr)\to \mathcal{F}\bigl(k(x_n)((x_1))\ldots ((x_{n-1}))\bigr)
$$
is surjective, which completes the proof.

\end{proof}

\section{$K_1^G$ of Laurent polynomials and power series over general rings}

\subsection{Results over general rings.}

In the present section we discuss various relations between  $K_1^G\bigl(R[[t]]\bigr)$, $K_1^G\bigl(R[t,t^{-1}]\bigr)$
and $K_1^G\bigl(R((t))\bigr)$, where $R$ is an arbitrary commutative ring and $G$ is a reductive algebraic  group
 defined over $R$. Our keystone result is the following theorem.

\begin{thm}\label{thm:E-decomp}
Let $R$ be a commutative ring, and let $G$ be a reductive group scheme over $R$, such that every semisimple normal subgroup
of $G$ contains $(\Gm_{,R})^2$. Then
$$
E\bigl(R((t))\bigr)=E\bigl(R[[t]]\bigr)E\bigl(R[t,t^{-1}]\bigr).
$$
\end{thm}

The proof of this theorem uses the notions of relative roots and relative root subschemes of reductive groups
introduced by V. Petrov and the author in~\cite{PS}. Their definitions and a sketch of construction
are given in \S~\ref{ssec:rel} below, and after that we give a proof of Theorem~\ref{thm:E-decomp}.
As for now, we discuss several easy corollaries of this theorem.  We begin with a reformulation of Theorem~\ref{thm:E-decomp} in terms of non-stable $K_1$-functors.

\begin{cor}
Let $R,G$ be as in Theorem~\ref{thm:E-decomp}. Then the sequence of pointed sets
$$
1\longrightarrow K_1^G(R[t])\xrightarrow{\st g\mapsto (g,g)} K_1^G(R[[t]])\times K_1^G(R[t,t^{-1}])
\xrightarrow{\st (g_1,g_2)\mapsto g_1{g_2}^{-1}} K_1^G\bigl(R((t))\bigr)
$$
is exact.
\end{cor}
\begin{proof}
Follows immediately from Theorem~\ref{thm:E-decomp} and Lemma~\ref{lem:f}.
\end{proof}

\begin{cor}
Let $R,G$ be as in Theorem~\ref{thm:E-decomp}. Then the natural homomorphism
$$
K_1^G(R[[t]])\to K_1^G(R((t)))
$$
is injective.
\end{cor}
\begin{proof}
Assume that the class of $g\in G(R[[t]])$ trivializes
in $K_1^G\bigl(R((t))\bigr)$, that is, $g\in G(R[[t]])\cap E\bigl(R((t))\bigr)$. Then by Theorem~\ref{thm:E-decomp}
we can assume that $g\in G(R[[t]])\cap E(R[t,t^{-1}])$. Since
\begin{equation}\label{eq:tt-1cap}
G(R[t,t^{-1}])\cap G(R[[t]])=G(R[t]),
\end{equation}
we have $g\in G(R[t])\cap E(R[t,t^{-1}])$. By Lemma~\ref{lem:f} this implies that
 $g\in E(R[t])$. Hence
$g\in E(R[[t]])$.
\end{proof}

The following corollary is what we use in the proof of Theorem~\ref{thm:main}.

\begin{cor}\label{cor:t-t-inj}
Let $R,G$ be as in Theorem~\ref{thm:E-decomp}. If $G(R[t])=G(R)E(R[t])$, then the natural homomorphism
$$
K_1^G(R[t,t^{-1}])\to K_1^G\bigl(R((t))\bigr)
$$
is injective.
\end{cor}
\begin{proof}
Assume that the class of $g\in G(R[t,t^{-1}])$ trivializes
in $K_1^G\bigl(R((t))\bigr)$, that is, $g\in G(R[t,t^{-1}])\cap E\bigl(R((t))\bigr)$. Then by Theorem~\ref{thm:E-decomp}
we can assume that $g\in G(R[t,t^{-1}])\cap E\bigl(R[[t]]\bigr)$. By~\eqref{eq:tt-1cap}
we have $g\in G(R[t])\cap E(R[[t]])$. Since
$G(R[t])=G(R)E(R[t])$, we can write $g=g_0g_1$ with $g_0\in G(R)$, $g_1\in E(R[t])$.
Since $g\in E(R[[t]])$, setting $t=0$ we deduce $g_0\in E(R)$. Hence $g\in E(R[t])\subseteq E(R[t,t^{-1}])$.
\end{proof}

\begin{rem}
The main result of~\cite{St-poly} shows that the equality $G(R[t])=G(R)E(R[t])$ holds if $R$ is a regular ring containing
a perfect field $k$, and $G$ is defined over $k$. Using other results of~\cite{St-poly}, Lemma~\ref{lem:f},
 and the techniques of~\cite{PaStV}, we can prove the same equality whenever $R$ is a regular ring containing an infinite field $k$, and $G$ is defined
over $R$. The latter result is still unpublished, so we decided not to use it in the present paper.
Instead, we give an independent and much simpler proof in the case where $R$ is a ring of Laurent polynomials,
and $G$ satisfies the same assumptions as in Theorem~\ref{thm:main}; see Lemma~\ref{lem:Lau-frac}.
\end{rem}

\subsection{Relative roots and relative root subschemes}\label{ssec:rel}

In order to prove Theorem~\ref{thm:E-decomp}, we need to use the notions of relative roots and relative root
subschemes. These notions were initially introduced and studied in~\cite{PS}, and further developed
in~\cite{St-thes}.
%Since the latter text has not been published in English, we reproduce here some
%statements and proofs.

Let $R$ be a commutative ring. Let $G$ be a reductive group scheme over $R$. Let $P$ be a parabolic subgroup
scheme of $G$ over $R$, and let $L$ be a Levi subgroup of $P$.
By~\cite[Exp. XXII, Prop. 2.8]{SGA3} the root system $\Phi$ of $G_{\overline{k(s)}}$, $s\in\Spec R$,
is constant locally in the Zariski topology on $\Spec R$. The type of the root system of
$L_{\overline{k(s)}}$ is determined by a Dynkin subdiagram
of the Dynkin diagram of $\Phi$, which is also constant Zariski-locally on $\Spec R$
by~\cite[Exp. XXVI, Lemme 1.14 and Prop. 1.15]{SGA3}. In particular, if $\Spec R$ is connected,
all these data are constant on $\Spec R$.

\begin{lem}\label{lem:relroots} %\cite[Gl. 1, Lemma 5]{St-thes}
Let $G$ be a reductive group over a connected commutative ring $R$, $P$ be a parabolic subgroup of $G$, $L$ be a Levi
subgroup of $P$, and $\bar L$ be the image of $L$ under the natural
homomorphism $G\to G^{\ad}\subseteq \Aut(G)$. Let $D$ be the Dynkin diagram of the root system $\Phi$ of
$G_{\overline{k(s)}}$ for any $s\in\Spec A$. We identify $D$ with a set of simple roots of $\Phi$. Let $J\subseteq D$
be the set of simple roots such that $D\setminus J\subseteq D$ is the subdiagram corresponing to $L_{\overline{k(s)}}$.
Then there are a unique maximal split subtorus
$S\subseteq\Cent(\bar L)$ and a subgroup $\Gamma\le \Aut(D)$ such that $J$ is invariant under $\Gamma$,
and
for any $s\in\Spec R$ and any split maximal torus $T\subseteq\bar L_{\overline{k(s)}}$
the kernel of the natural surjection
\begin{equation}\label{eq:T-S}
X^*(T)\cong\ZZ\Phi\xrightarrow{\ \pi\ } X^*(S_{\overline{k(s)}})\cong \ZZ\Phi(S,G)
\end{equation}
is generated by all roots $\alpha\in D\setminus J$,
and by all differences $\alpha-\sigma(\alpha)$, $\alpha\in J$, $\sigma\in\Gamma$.
\end{lem}
\begin{proof}
%\end{proof}

%In the setting of Lemma~\ref{lem:relroots-cover}, assume that already the base ring $R=A$ satisfies
%all the conditions (1)-(4). Denote by $\Phi$ the root system of $G$, by $\Pi$ a set of simple roots of $\Phi$, by $D$
%the corresponding Dynkin diagram, and by $\Gamma$  the subgroup of
%$\Aut D$ correspoding to the $*$-action. Let $J$ be the subset of $\Pi$ such that $\Pi\setminus J$ is the type
%of $P_{\overline{k(s)}}$ (that is, the set of simple roots of the Levi sugroup $L_{\overline{k(s)}}$)
%for all $s\in\Spec A$.
%Then $J$ is $\Gamma$-invariant.

%\begin{lem}{lem:relroots-pres}
%Let $A'/A$ be a ring extension such that $T$ splits over $A'$. We have
%$X^*(T_{A'})\cong\ZZ\Phi$.
%Let
%\end{lem}
%\begin{proof}
We can assume that $G=G^{\ad}$ from the start, and $L=\bar L$.
The radical $\rad(L)=\Cent(L)^\circ$ of $L$ is a torus. Since $\Spec R$ is connected, it contains a unique
maximal split subtorus $S\subseteq \Cent(L)$ by~\cite[Exp. XXVI, 6.5]{SGA3}. In order to show that
the kernel of the map~\eqref{eq:T-S} is as required, we use the notion of the Dynkin scheme of $G$.

By construction~\cite[Exp. XXIV, \S 3.7]{SGA3}, the Dynkin scheme $\Dyn(G)$ over $R$ is
an \'etale twisted form of the constant Dynkin scheme $D_R$ over $R$. It is thus
a finite \'etale scheme over $R$ endowed with a subscheme $E\subseteq\Dyn(G)\times_R\Dyn(G)$ not intersecting
the diagonal (the scheme of edges of the Dynkin diagram) and a morphism $\Dyn(G)\to\{1,2,3\}_R$ (the lengths of simple roots). Clearly,
there is a finite Galois ring extension $R'/R$ such that $\Dyn(G)_{R'}\cong D_{R'}$ is split.
Since $\Spec R$ is connected, the scheme $\Dyn(G)$ is uniquely
determined by $D$ together with a subgroup $\Gamma$ of $\Aut(D)$
that represents the action of the Galois
group $\Gal(R'/R)$ on $D$. The orbits of $\Gamma$ in $D$ are in one-to-one correspondence with
minimal clopen $R$-subschemes of $\Dyn(G)$. The parabolic subgroup $P$ of $G$ is defined over $R$, hence
by~\cite[Exp. XXVI, \S 3]{SGA3} $\Dyn(G)$ contains a clopen $R$-subscheme $t(P)$, called the type of $P$,
which is a twisted form of $J_R\subseteq D_R$. In particular, $J$ is a $\Gamma$-invariant subset of $D$.
The subscheme $\Dyn(G)\setminus t(P)$ is the twisted form of $(D\setminus J)_R$ isomorphic to the Dynkin scheme $\Dyn(L)$.

%Let $\Dyn(G)$ be the Dynkin scheme defined in~\cite[Exp. XXIV, 3.3, 3.7]{SGA3}.
%The Dynkin scheme is an \'etale twisted form of the constant Dynkin scheme $D_R$ over $R$.
%Since the type of $G$ is constant on $\Spec A$ for any $A=A_i$, $1\le i\le m$,
%by~\cite[Exp. XXIV, Corollaire 1.19]{SGA3} there is a cocycle class $\xi\in H^1_{\et}(A,\Aut(G_0))$,
%where $G_0$ is the split adjoint group over $A$ of type $D$, such that $G_A$ is the $\xi$-twisted form
%of $A$. Then $\Dyn(G)_A$ is the twisted form of $D_A$ by the image $\xi_D$ under the map
%$$
%H^1_{\et}(A,\Aut(G_0))\to H^1_{\et}(A,\Aut(D)),
%$$
%induced by~\eqref{eq:Aut-D}.

Recall that there exists a quasi-split
reductive group $G_{\qs}$ over $R$ of the same type as $G$ (in particular, adjoint), such that $G$
is an inner twisted form of $G_{\qs}$, that is, $G$ is given by a cocycle class in
$H^1_{\et}(R,G_{\qs})$~\cite[Exp. XXIV, 3.12]{SGA3}. One also has $\Dyn(G)\cong\Dyn(G_{qs})$~\cite[Th\'eor\`eme 3.11]{SGA3}.
Let $P_{\qs}$ be a parabolic subgroup in $G_{\qs}$ of the same type
as $P$ that is standard, i.e. contains a Killing couple $T_{\qs}\subseteq B_{\qs}$, and let $L_{\qs}$
be the standard Levi subgroup of $P_{\qs}$ containing $T_{\qs}$.

First, we study the torus $S$ in the case where $G=G_{\qs}$, $P=P_{\qs}$, and $L=L_{\qs}$.
%By~\cite[Exp. XXVI, 3.21]
%{SGA3} the cocycle $\eta$ is in the image of $H^1_{\et}(A,P_{\qs})$.
%By~\cite[Exp. XXVI, Corollaire 2.3]{SGA3} the natural map $H^1_{\et}(A,L_{\qs})\to H^1_{\et}(A,P_{\qs})$
%is bijective.
%Since the normalizer of the couple $(P,L)$ in $G$ coincides with $L$, and
%the desired torus $S$ should be central in $L$, it is enough to construct $S$
%in the case where $G\cong G_{\qs}$.
%Let $P_{\qs}$ be a parabolic subgroup in $G_{\qs}$ of the same type
%as $P$ that is standard, i.e. contains a Killing couple $T_{\qs}\subseteq B_{\qs}$, and let $L_{\qs}$
%be the standard Levi subgroup of $P_{\qs}$ containing $T_{\qs}$. Since $P$ and $P_{\qs}$ are of the same
%type, they are locally conjugate. Therefore, it is enough to construct $S$ for $P=P_{\qs}$.
%The Dynkin scheme $\Dyn(G_{\qs})=\Dyn(G)$ is a disjoint union of its
%minimal clopen subschemes which correspond to orbits of the subgroup $\Gamma\subseteq\Aut(D)$ in
%$D$.
There is an explicit presentation of $T_{\qs}$ as a product of Weil restrictions of $\Gm$~\cite[Exp. XXIV Prop. 3.13]{SGA3}:
$$
T_{\qs}\cong R_{\Dyn(G_{\qs})/R}\bigl({\Gm}_{,\Dyn(G_{\qs})}\bigr)\cong \prod_O R_{O/R}(\Gm_{,O}),
$$
where $O$ runs over all minimal clopen subschemes of $\Dyn(G)$. This presentation is obtained
by descent from the standard decomposition of a split maximal torus into a direct product
of 1-dimensional tori corresponding to the vertices of $D$.
%Recall that the type $t(P_{\qs})$ is a clopen
%subscheme of $\Dyn(G)$ which is a twisted form of $J_A\subseteq D_A$~\cite[Exp. XXVI, D\'ef. 3.4]{SGA3}.
By~\cite[Prop. 1 (2)]{PS-tind}
we have
$$
\Cent(L_{\qs})\cong\prod\limits_{O\not\subseteq t(P)} R_{O/R}({\Gm}_{,O}),
$$
where $O$
runs over all minimal clopen subschemes of $\Dyn(G)$ not contained in $t(P)$. Then, clearly,
\begin{equation}\label{eq:S-prod}
S=\prod_{O\not\subseteq t(P)}{\Gm}_{,R}\subseteq\Cent(L_{\qs}),
\end{equation}
where each ${\Gm}_{,R}$ is the canonical split subtorus of $R_{O/R}({\Gm}_{,O})$.
% Since the subgroup $\Gamma\subseteq \Aut(D)$ is constant on $\Spec A$, the dimension
%of $S$ is constant on $\Spec A$, and hence it is a split torus.
%By~\cite[Prop. 1 (3)]{PS-tind} we have $L_{\qs}=\Cent_{G_{\qs}}(S)$. Hence
% we have $\Pi\setminus J\subseteq\ker\pi$. Since the torus $S$ is defined over $A$, its action on $\Lie(G)$ commutes with
For any $\pi$ as in the statement of the lemma, by~\eqref{eq:S-prod} all roots $\alpha\in D\setminus J$,
and all differences $\alpha-\sigma(\alpha)$, $\alpha\in J$, $\sigma\in\Gamma$, belong to $\ker\pi$.
Since the rank of $X^*(S)$ is equal to the number of orbits of $\Gamma$ in $D\setminus J$, these elements
generate $\ker\pi$.

Now we consider $S$ in the general case where $G\neq G_{\qs}$. Let $\eta\in Z^1_{\et}(R,G_{\qs})$ be a cocycle corresponding
to the twisted form $G$ of $G_{\qs}$. Let $\bigsqcup\Spec R_\tau\to\Spec R$ be an \'etale covering
(which we can and do assume to be affine for simplicity) such that
$G_{R_\tau}\cong (G_{\qs})_{R_\tau}$ for each $\tau$, and let $g_{\sigma\tau}\in G_{\qs}(R_\tau\otimes_R R_\sigma)$
be the elements representing $\eta$ on this covering. For each $\tau$, the pair $(L_{R_\tau},P_{R_\tau})$,
considered inside $(G_{\qs})_{R_\tau}$, is conjugate to the pair $((L_{\qs})_{R_\tau},(P_{\qs})_{R_\tau})$
locally in the \'etale topology on $\Spec R_\tau$ by~\cite[Exp.XXVI, 4.5.2]{SGA3}. Refining our \'etale covering,
we can assume that these pairs are conjugate already over $R_\tau$, i.e.
$L_{R_\tau}=f_\tau(L_{\qs})_{R_\tau}f_\tau^{-1}$ and $P_{R_\tau}=f_\tau(P_{\qs})_{R_\tau}f_\tau^{-1}$
 for an element $f_\tau\in G_{\qs}(R_\tau)$.
Note that $g_{\sigma\tau}$ preserves the pair
$(L_{R_\tau\otimes_R R_\sigma},P_{R_\tau\otimes_R R_\sigma})$, since $L$ and $P$ are defined over $R$.
%and the same is true for $f_{\sigma}f_\tau^{-1}$.
Since the normalizer of
$((L_{\qs})_{R_\tau\otimes_R R_\sigma},(P_{\qs})_{R_\tau\otimes_R R_\sigma})$ in
$(G_{\qs})_{R_\tau\otimes_R R_\sigma}$ by~\cite[Exp. XXVI Prop. 1.2, Prop. 1.6]{SGA3} equals
$(L_{\qs})_{R_\tau\otimes_R R_\sigma}$, we conclude that
\begin{equation}\label{eq:fgsigmatau}
f_{\sigma}^{-1}g_{\sigma\tau}f_\tau\in L_{\qs}(R_\tau\otimes_R R_\sigma).
\end{equation}
Let $S_{\qs}$ be the maximal split $R$-subtorus of $\Cent(L_{\qs})$.
Since $S_{\qs}$ is central in $L_{\qs}$, by~\eqref{eq:fgsigmatau} we have that
$$
g_{\sigma\tau}f_\tau|_{(S_{\qs})_{R_\tau\otimes_R R_\sigma}}=f_\sigma|_{(S_{\qs})_{R_\tau\otimes_R R_\sigma}}
$$
as $R_\tau\otimes_R R_\sigma$-group scheme morphisms from $(S_{\qs})_{R_\tau\otimes_R R_\sigma}$ to
$\Cent(L)_{R_\tau\otimes_R R_\sigma}$ induced by conjugation.
By faithfully
flat descent for affine morphisms~\cite[Part 1, Theorem 4.33]{FGA-ex}, there is a closed embedding
of $R$-group schemes $i:S_{\qs}\to \Cent(L)$ such that $i_{R_\tau}=f_\tau|_{(S_{\qs})_{R_\tau}}$ for each $\tau$. Clearly,
$i(S_{\qs})$ is contained in the maximal split subtorus $S$ of $\Cent(L)$.

Note that we can interchange the roles of the groups $(G_{\qs},P_{\qs},L_{\qs})$ and $(G,P,L)$ in the argument of
the previous paragraph. Indeed, since $G$ is given by a cocycle with values in $G_{\qs}$, coversely,
$G_{\qs}$ is given by a cocycle with values in $G$; cf.~\cite[Ch. I, Proposition 35]{Se}.
Then we conclude that there is a closed embedding of $R$-group schemes $j:S\to \Cent(L_{\qs})$ as well.
Therefore, $i(S_{\qs})=S$, since these tori have the same rank.
It remains to note that, since
$\Dyn(G_{\qs})=\Dyn(G)$ and for any
$s\in\Spec R$ the homomorphism $R\to\overline{k(s)}$ factors through one of the rings
$R_\tau$, the torus $S\subseteq\Cent(L)$ satisfies the claim of the lemma on $\ker\pi$, since $S_{\qs}$ does.
\end{proof}

In~\cite{PS}, we introduced a system of relative roots $\Phi_P$ with respect to a parabolic
subgroup $P$ of a reductive group $G$ over a commutative ring $R$. This system $\Phi_P$ was defined
independently over each member $\Spec A=\Spec A_i$
of a suitable finite disjoint Zariski covering
$$
\Spec R=\coprod\limits_{i=1}^m\Spec A_i,
$$
such that over each $A=A_i$, $1\le i\le m$, the root system $\Phi$ and the Dynkin diagram $D$ of $G$ is constant.
Namely, we considered the formal projection
$$
\pi_{J,\Gamma}\colon\ZZ \Phi
\longrightarrow \ZZ\Phi/\l<D\setminus J;\ \alpha-\sigma(\alpha),\ \alpha\in J,\ \sigma\in\Gamma\r>,
$$
and set $\Phi_P=\Phi_{J,\Gamma}=\pi_{J,\Gamma}(\Phi)\setminus\{0\}$. The last claim of Lemma~\ref{lem:relroots}
allows to identify $\Phi_{J,\Gamma}$ and $\Phi(S,G)$ whenever $\Spec R$ is connected.

\begin{defn}
In the setting of Lemma~\ref{lem:relroots} we call $\Phi(S,G)$ a \emph{system
of relative roots with respect to the parabolic subgroup $P$ over $R$} and denote it by $\Phi_P$.
\end{defn}

\begin{exa}\label{exa:roots}
If $A$ is a field or a local ring, and $P$ is a minimal parabolic subgroup of $G$,
then $\Phi_P$ is nothing but the relative root system of $G$ with respect to a maximal split subtorus
in the sense of~\cite{BoTi} or, respectively,~\cite[Exp. XXVI \S 7]{SGA3}.
\end{exa}

We have also defined in~\cite{PS} irreducible components of systems of relative roots, the subsets of positive and negative
relative roots, simple relative roots, and the height of a root. These definitions are immediate analogs
of the ones for usual abstract root systems, so we do not reproduce them here.

%We also constructed in~\cite{PS}, for any relative root $\alpha\in\Phi_P$
%a finitely generated projective $A$-module $V_\alpha$ and a closed embedding
%$$
%X_\alpha: W(V_\alpha)\to G,
%$$
%where $W(V_\alpha)$ is the affine group scheme over $A$ defined by $V_\alpha$.
%These subschemes $X_\alpha$, $\alpha\in\Phi_P$, are called relative root subschemes of $G$, and
%possess several nice properties
%similar to that of elementary root subgroups of a split reductive group.

%\cite[Gl. 1, Lemma 12, Lemma 13]{St-thes};\cite[Th. 2]{PS}

Let $R$ be a commutative ring with 1.
For any finitely generated projective $R$-module $V$, we denote by $W(V)$ the natural affine scheme
over $R$ associated with $V$, see~\cite[Exp. I, \S 4.6]{SGA3}.
%$S\mapsto V\otimes_RS$ is represented by an affine group scheme
%$\WW(V)=\Spec\Sym^*(V^*)$, where $V^*$ is the dual $R$-module, $\Sym^*$ is the symmetric algebra.
Any morphism of $R$-schemes $\WW(V_1)\to\WW(V_2)$
is determined by an element $f\in\Sym^*(V_1^\vee)\otimes_R V_2$, where $\Sym^*$ denotes the symmetric algebra,
and $V_1^\vee$ denotes the dual module of $V_1$. If $f\in\Sym^d(V_1^\vee)\otimes_R V_2$,
we say that the corresponding morphism is
homogeneous of degree $d$.
%In particular, morphisms of degree $1$ are linear morphisms.
By abuse of notation, we also write $f:V_1\to V_2$ and call it {\it a degree $d$
homogeneous polynomial map from $V_1$ to $V_2$}. In this context, one has
$$
f(\lambda v)=\lambda^d f(v)
$$
for any $v\in V_1$ and $\lambda\in R$.

\begin{lem}\label{lem:relschemes}\cite{PS}
In the setting of Lemma~\ref{lem:relroots}, for any $\alpha\in\Phi_P=\Phi(S,G)$ there exists a closed
$S$-equivariant  embedding of $R$-schemes
$$
X_\alpha\colon W\bigl(\Lie(G)_\alpha\bigr)\to G,
$$
satisfying the following condition.

\begin{itemize}
\item[\bf{($*$)}] Let $R'/R$ be any ring extension such that $G_{R'}$ is split with
respect to a maximal split $R'$-torus $T\subseteq L_{R'}$. Let $e_\delta$,
$\delta\in\Phi$, be a Chevalley basis of $\Lie(G_{R'})$, adapted to $T$ and $P$, and $x_\delta\colon\Ga\to G_{R'}$, $\delta\in\Phi$, be the associated
system of 1-parameter root subgroups
{\rm(}e.g. $x_\delta=\exp_\delta$ of~\cite[Exp. XXII, Th. 1.1]{SGA3}{\rm)}.
Let
$$
\pi:\Phi=\Phi(T,G_{R'})\to\Phi_P\cup\{0\}
$$
be the natural projection.
Then for any
$u=\hspace{-8pt}\sum\limits_{\delta\in\pi^{-1}(\alpha)}\hspace{-8pt}a_\delta e_\delta\in\Lie(G_{R'})_\alpha$
%\hspace{-1ex}\sum_{\substack{-2n \leq i \leq 2n \\[.7ex] i\neq 0}}
one has
\begin{equation}\label{eq:Xalpha-prod}
X_\alpha(u)=
\Bigl(\prod\limits_{\delta\in\pi^{-1}(\alpha)}\hspace{-8pt}x_{\delta}(a_\delta)\Bigr)\cdot
\prod\limits_{i\ge 2}\Bigl(\prod\limits_{
\st
\theta\in \pi^{-1}(i\alpha)
}
\hspace{-8pt}x_\theta(p^i_{\theta}(u))\Bigr),
\end{equation}
where every $p^i_{\theta}:\Lie(G_{R'})_\alpha\to R'$ is a homogeneous polynomial map of degree $i$,
and the products over $\delta$ and $\theta$ are taken in any fixed order.
\end{itemize}
\end{lem}
\begin{proof}
Proceeding exactly as in~\cite[Th. 2]{PS}, we prove the existence of a closed embedding
$$
X_\alpha\colon W(V_\alpha)\to G
$$
satisfying condition $(*)$, where $V_\alpha$ is a finitely generated projective $R$-module of rank $|\pi^{-1}(\alpha)|$,
implicitly constructed by descent. However, once we identify the system of relative roots $\Phi_P$
in the sense of~\cite{PS} with $\Phi(S,G)$ as discussed above, it follows from the proof of~\cite[Th. 2]{PS} that
$V_\alpha$ is canonically isomorphic to $\Lie(G)_\alpha$. The $S$-equivariance of $X_\alpha$ follows
immediately from condition $(*)$.
\end{proof}

\begin{defn}
Closed embeddings $X_\alpha$, $\alpha\in\Phi_P$, satisfying the statement of Lemma~\ref{lem:relschemes},
are called \emph{relative root subschemes of $G$ with respect to the parabolic subgroup $P$}.
\end{defn}

\begin{rem}
Relative root subschemes of $G$ with respect to $P$, actually,
depend on the choice of a Levi subgroup $L$ in $P$, but their essential properties stay the same,
so we usually omit $L$ from the notation.
\end{rem}

\begin{exa}\label{exa:exp}
Let $A$ be a connected commutative ring that contains $\QQ$, and let $G$
be a semisimple reductive group of adjoint type over $A$ containig  a parabolic subgroup $P$ with a Levi subgroup $L$.
We identify $G$ with its image under
the natural homomorphism $G\to\Aut_A(\Lie(G))$.
%First we recall that since $G$ is adjoint, we have $\Cent(G)=1$, and the natural homomorphism $G\to\Aut_A(\Lie(G))$ is a closed embedding, e.g.
%by~\cite[Exp. II, p. 78, Corollaire 1; Exp. XVI, Corollaire 1.5 a)]{SGA3}. Thus we can indeed
%identify $G$ with its image in $\Aut_A(\Lie(G))$.
Then the relative root $A$-subschemes $X_\alpha$, $\alpha\in\Phi_P$, of
Lemma~\ref{lem:relschemes} can be constructed as follows.
For any ring extension $B/A$ and any $v\in V_\alpha\otimes_A B=\Lie(G_B)_\alpha$ set
$$
X_\alpha(v)=\exp(\ad_v)=\sum\limits_{i=0}^\infty\frac 1{i!}(\ad_v)^i\in\Aut_B(\Lie(G_B)).
$$
Here the "infinite"{} sum is necessarily finite, since we have $(\ad_v)^i=0$ for any $i>|\Phi_P|$.
%Standard formulas imply that
%$\exp(ad_v)\in\Aut_A(\Lie(G))$.
It is clear that $X_\alpha:W(\Lie(G)_\alpha)\to\Aut_A(\Lie(G))$ is a morphism of $A$-schemes.

Apart from that, we need to show that each $X_\alpha$ is an $S$-equivariant closed embedding and
satisfies the condition $(*)$ of Lemma~\ref{lem:relschemes}.
Let $A'/A$ be any ring extension such that $G_{A'}$ is split with
respect to a maximal split $A'$-torus $T\subseteq L_{A'}$.
We recall that $x_\delta(t)$, $\delta\in\Phi$, $t\in A'$, coincide
 with $\exp(t\ad_{e_\delta})$ in the adjoint representation of $G_{A'}$, e.g.~\cite{Chev}.
% unicity of root subgroups Exp. XX Cor. 1.20
Then the Baker--Campbell--Hausdorff formula
implies that~\eqref{eq:Xalpha-prod} holds, and that each morphism
$$
X_\alpha:W(\Lie(G_{A'})_\alpha)\to G_{A'}
$$
is $S_{A'}$-equivariant. Denote by $X_\alpha^H$ the morphism $X_\alpha$ considered
as a morphism from $W(\Lie(G_{A'})_\alpha)$ to the unipotent closed $A'$-subgroup
$$
H=\prod\limits_{i\ge 1}\prod\limits_{\delta\in\pi^{-1}(i\alpha)}x_\delta(\Ga)\cong
W\Bigl(\sum_{i\ge 1}\Lie(G_{A'})_{i\alpha}\Bigr)
$$
of $G_{A'}$.
Then~\eqref{eq:Xalpha-prod} readily implies that $X_\alpha^H$ is universally closed and universally injective,
and hence the same is true for $X_\alpha$.
Since the tangent map $\Lie(X_\alpha)$, corresponding to the inclusion of $\Lie(G_{A'})_\alpha$
into $\Lie(G_{A'})$, is also injective,
%$$
%\Lie(X_\alpha):\Lie(W(\Lie(G_{A'})_\alpha))=\Lie(G_{A'})_\alpha\to\Lie(G)
%$$
we conclude that
%the sheaf of relative differentials corresponding to $X_\alpha$ is trivial, and hence
$X_\alpha$ is formally unramified. Summing up, this implies that
$X_\alpha$ is a closed embedding of $W(\Lie(G_{A'})_\alpha)$ into $G_{A'}$.

Finally, note that locally in the \'etale topology, the group $G$ over $A$ is split with respect to a torus $T$ contained in
$L\subseteq P$, see~\cite[Exp. XXII, Cor. 2.3; Exp. XXVI, Lemme 1.14]{SGA3}.
Then faithfully flat descent implies that in order to prove that $X_\alpha$ is an
$S$-equivariant closed embedding over $A$, it is enough to prove the same over every $A'$ as above.
Since the latter is already established, we conclude that $X_\alpha$ satisfies all the conditions present in
Lemma~\ref{lem:relschemes}.

%and the fact that $X_\alpha$ is an $S$-equivariant closed embedding over each
%$A'$ as in  the condition $(*)$, implies that $X_\alpha$ is an $S$-equivariant closed embedding over $A$.

%Thus, there is an $fpqc$-covering of $\Spec A$ consisting of $A'/A$
%as in the condition $(*)$ of Lemma~\ref{lem:relschemes}, the morphism
%$$
%X_\alpha:W(\Lie(G_{A'})_\alpha)\to G_{A'}
%$$
%is an $S_{A'}$-equivariant closed embedding, then we conclude
%in order to prove that $X_\alpha$ is an
%$S$-equivariant closed embedding, it is enough to do that over every $A'$ as in the condition $(*)$
% of Lemma~\ref{lem:relschemes}.

\end{exa}

We will use the following properties of relative root subschemes.

\begin{lem}\label{lem:rootels}\cite[Theorem 2, Lemma 6, Lemma 9]{PS}
Let $X_\alpha$, $\alpha\in\Phi_P$, be as in Lemma~\ref{lem:relschemes}.
Set $V_\alpha=\Lie(G)_\alpha$ for short. Then

(i) There exist degree $i$ homogeneous polynomial maps $q^i_\alpha:V_\alpha\oplus V_\alpha\to V_{i\alpha}$, i>1,
such that for any $R$-algebra $R'$ and for any
$v,w\in V_\alpha\otimes_R R'$ one has
\begin{equation}\label{eq:sum}
X_\alpha(v)X_\alpha(w)=X_\alpha(v+w)\prod_{i>1}X_{i\alpha}\l(q^i_\alpha(v,w)\r).
\end{equation}
%where each $q^i_\alpha\colon\WW(V_\alpha)\times_{\Spec R}\WW(V_\alpha)=\WW(V_\alpha\oplus V_\alpha)
%\to\WW(V_{i\alpha})$ is a homogeneous map of degree $i$.

(ii) For any $g\in L(R)$, there exist degree $i$ homogeneous polynomial maps
$\varphi^i_{g,\alpha}\colon V_\alpha\to V_{i\alpha}$, $i\ge 1$, such that for any $R$-algebra $R'$ and for any
$v\in V_\alpha\otimes_R R'$ one has
$$
gX_\alpha(v)g^{-1}=\prod_{i\ge 1}X_{i\alpha}\l(\varphi^i_{g,\alpha}(v)\r).
$$
If $g\in S(R)$, then $\varphi^1_{g,\alpha}$ is multiplication by a scalar, and
all $\varphi^i_{g,\alpha}$, $i>1$, are trivial.

(iii) \emph{(generalized Chevalley commutator formula)} For any $\alpha,\beta\in\Phi_P$
such that $m\alpha\neq -k\beta$ for all $m,k\ge 1$,
there exist polynomial maps
$$
N_{\alpha\beta ij}\colon V_\alpha\times V_\beta\to V_{i\alpha+j\beta},\ i,j>0,
$$
homogeneous of degree $i$ in the first variable and of degree $j$ in the second
variable, such that for any $R$-algebra $R'$ and for any
for any $u\in V_\alpha\otimes_R R'$, $v\in V_\beta\otimes_R R'$ one has
\begin{equation}\label{eq:Chev}
[X_\alpha(u),X_\beta(v)]=\prod_{i,j>0}X_{i\alpha+j\beta}\bigl(N_{\alpha\beta ij}(u,v)\bigr)
\end{equation}

(iv) For any subset $\Psi\subseteq X^*(S)\setminus\{0\}$ that is closed under addition,
the morphism
$$
X_\Psi\colon W\Bigl(\,\bigoplus_{\alpha\in\Psi}V_\alpha\Bigr)\to U_\Psi,\qquad
(v_\alpha)_\alpha\mapsto\prod_\alpha X_\alpha(v_\alpha),
$$
where the product is taken in any fixed order,
is an isomorphism of schemes.
\end{lem}

\subsection{Proof of Theorem~\ref{thm:E-decomp}}

\begin{proof}[Proof of Theorem~\ref{thm:E-decomp}.]

By assumption, every semisimple normal subgroup of $G$ contains $(\Gm_{,R})^2$. We
claim that there is a split subtorus $S_0$ of $G$ such that $S_0\cap H$
contains $(\Gm_{,R})^2$ for every semisimple normal subgroup $H$ of $G$. Indeed, if $G$ is semisimple
and simply connected, this follows from the fact that $G$ is a direct product of its minimal normal
semisimple subgroups~\cite[Exp. XXIV, \S 5]{SGA3}. In general,
$G$ is a quotient of the direct product $G^{\scl}\times\rad(G)$ by a central finite
subgroup, where $\rad(G)$ is the radical
of $G$, and $G^{\scl}$ is the simply connected cover of the derived group scheme of $G$.
One readily sees that if $S_0^{\scl}$ is the split subtorus of $G^{\scl}$ whose intersection
with every semisimple normal subgroup of $G^{\scl}$ contains $(\Gm_{,R})^2$, then the same is true
for its image $S_0$ in $G$.

%Then $G$ is a quotient of the direct product $\mathrm{der}(G)^{\scl}\times\rad(G)$ by a central finite
%subgroup, where
%$\mathrm{der}(G)^{\scl}$ is the simply connected covering of the semisimple group $\mathrm{der}(G)$.
%By~\cite[Exp. \S 5]{SGA3}, the group $\mathrm{der}(G)^{\scl}$ is a direct product of its minimal normal
%semisimple subgroups.

 Let $P$ be a parabolic subgroup
of $G$ with a Levi subgroup $L=\Cent_G(S_0)$ constructed as in Remark~\ref{rem:S-P}. Clearly, $P$ is
strictly proper. Let $P^-$ be the opposite to $P$ parabolic subgroup satisfying $L=P\cap P^-$. For any $R$-algebra $A$, we have
$$
E(A)=\l<U_P(A),U_{P^-}(A)\r>.
$$

Now we show that, in order to prove the equality
\begin{equation}\label{eq:EEE}
E\bigl(R((t))\bigr)=E\bigl(R[[t]]\bigr)E\bigr(R[t^{\pm 1}]\bigr),
\end{equation}
we can assume that $R$ is connected. Fix an element $g\in E\bigl(R((t))\bigr)$.
The commmutative ring $R$ is a direct limit
of its Noetherian subrings, $R=\varinjlim R_\alpha$. Fix an element $g\in E\bigl(R((t))\bigr)$.
Since $G$, its semisimple normal subgroup schemes, $P$, $P_-$ and $L$
are all finitely presented $R$-group schemes, there is an index $\alpha$ and
a reductive group scheme $G'$ over $R_\alpha$ such that all these group schemes are defined over $R_\alpha$,
$P$ is strictly proper over $R_\alpha$,
and $g\in E\bigl(R_\alpha((t))\bigr)\le E\bigl(R((t))\bigr)$; cf.~\cite[Exp. XIX, Remarque 2.9]{SGA3}.
Clearly, in order to show that $g$ belongs to the right-hand side of~\eqref{eq:EEE}, it is enough to prove
the equality~\eqref{eq:EEE} for the Noetherian ring $R_\alpha$ in place of $R$.
Thus, we can assume from the start that $R$ is Noetherian. Then $R=\prod\limits_{i=1}^m A_i$, where
$A_i$, $1\le i\le m$, are connected rings.
Set $R_1=R((t))$, $R_2=R[t^{\pm 1}]$, and $R_3=R[[t]]$. Then
$$
E(R_j)=\l<U_P(R_j),U_{P^-}(R_j)\r>=\prod_{i=1}^m E(A_i\otimes_R R_j)
$$ for all $R_j$, $j=1,2,3$.
Therefore, it is enough to show that~\eqref{eq:EEE} holds with $R$ replaced by each of the connected rings $A_i$.
Thus, we can assume from now on that $R$ is connected.

Let $S\subseteq \Cent(\bar L)$ be the split torus constructed in Lemma~\ref{lem:relroots}, $\Phi_P=\Phi(S,G)$,
and $X_\alpha$, $\alpha\in \Phi_P$, be the relative root subschemes over $R$ that exist by Lemma~\ref{lem:relschemes}.
Since, clearly, $S$ contains the image of $S_0$ in $G^{\ad}$, every irreducible component of $\Phi_P$ in the sense of~\cite{PS}
has rank $\ge 2$.

By Lemma~\ref{lem:rootels} (iv) the group $E(R((t)))$ is generated by root elements
$X_\alpha(v)$, $\alpha\in\Phi_P$, $v\in V_\alpha\otimes_R R((t))$. Write $v=\sum\limits_{j=-k}^\infty v_jt^j$, $v_j\in V_\alpha$, $k\ge 0$.
By the equality~\eqref{eq:sum} of Lemma~\ref{lem:rootels} we have
$$
X_\alpha(v)=X_\alpha\left(\sum\limits_{j=-k}^0v_jt^j\right)X_\alpha\l(\sum\limits_{j=1}^\infty v_jt^j\r)
\prod\limits_{i\ge 2}X_{i\alpha}(u_i)$$
for some $u_i\in V_{i\alpha}\otimes_R R((t))$. Applying induction on the height of $\alpha$, we conclude
that $X_\alpha(v)$ decomposes into a product of elements from $E(R[t^{-1}])$ and $E(R[[t]])$, that is,
$$
E\bigl(R((t))\bigr)=\l<E(R[t^{-1}]),E(R[[t]])\r>.
$$
Similarly, one concludes that $E(R[t,t^{-1}])$ is generated by elements $X_\alpha(t^n u)$, $n\in\ZZ$,
$u\in V_\alpha$, $\alpha\in\Phi_P$.
Consequently, in order to prove~\eqref{eq:EEE},
%it is enough to show that
%$$
%E(R[t,t^{-1}])E(R[[t]])\subseteq E(R[[t]])E(R[t,t^{-1}]).
%$$
it is enough to show that
for any $\beta\in\Phi_P$, $v\in V_\beta\otimes_R R[[t]]$
we have
\begin{equation}\label{eq:[tt]-v}
E(R[t,t^{-1}])X_\beta(v)\subseteq E(R[[t]])E(R[t,t^{-1}]).
\end{equation}

For any $R$-algebra $R'$, any ideal $I\subseteq R'$, and any additively closed set
$\Psi\subseteq X^*(S)\setminus\{0\}$, we set
$$
U_\Psi(I)=\l<X_\alpha(u),\ \alpha\in\Psi,\ u\in V_\alpha\otimes_R I\r>\subseteq U_\Psi(R'),
$$
and
$$
E(I)=\l<X_\alpha(u),\ \alpha\in\Phi_P,\ u\in V_\alpha\otimes_R I\r>\subseteq E(R').
$$
We show that for any $\beta\in\Phi_P$, $v\in V_\beta\otimes_R R[[t]]$ one has
\begin{equation}\label{eq:betaN}
X_\beta(v)\in E(t^NR[[t]])E(R[t])\quad\mbox{for any}\ N\ge 0.
\end{equation}
More precisely, set $(\beta)=\{i\beta\ |\ i\ge 1\}$; we show that
\begin{equation}\label{eq:betaN-2}
X_\beta(v)\in U_{(\beta)}(t^NR[[t]])\cdot U_{(\beta)}(R[t])
\end{equation}
arguing by descending induction on the height of $\beta$.
Since $V_\beta$ is a finitely generated projective $R$-module,
we can write $v=v_1+t^Nv_2$, where $v_1\in V_\beta\otimes_R R[t]$ and $v_2\in V_\beta\otimes_R R[[t]]$.
Then~\eqref{eq:sum} of Lemma~\ref{lem:rootels} implies that
\begin{equation}\label{eq:betaN-3}
X_\beta(v)=X_\beta(t^Nv_2)\cdot\prod_{i>1}X_{i\beta}(q_\beta^i(v,-v_1))\cdot (X_\beta(-v_1))^{-1}.
\end{equation}
By the induction hypothesis, for any $i>1$ one has
\begin{equation}\label{eq:betaN-4}
X_{i\beta}(q_\beta^i(v,-v_1))\in U_{(i\beta)}(t^NR[[t]])\cdot U_{(i\beta)}(R[t])\subseteq
U_{(\beta)}(t^NR[[t]])\cdot U_{(\beta)}(R[t]).
\end{equation}
Note that by~\eqref{eq:Chev} of Lemma~\ref{lem:rootels}
the group $U_{(\beta)}(R[t])$ normalizes the group $U_{(\beta)}(t^NR[[t]])$.
Then, clearly,~\eqref{eq:betaN-3} and~\eqref{eq:betaN-4} together imply~\eqref{eq:betaN-2}. This finishes the
proof of~\eqref{eq:betaN}.

%Any element $x\in E(R[t,t^{-1}])$ is of the form $x=\prod\limits_{j=1}^mX_{\beta_i}(t^{k_i}u_i)$ for some
%$\beta_1,\ldots,\beta_m\in\Phi_P$, $k_1,\ldots,k_m\in\ZZ$,
%and $u_i\in V_{\beta_i}$, $1\le i\le m$.
Next, we show that for any $n\in\ZZ$,
$u\in V_\alpha$, $\alpha\in\Phi_P$, and $M\ge 0$ there is $N\ge 0$ such that
\begin{equation}\label{eq:nNM}
X_\alpha(t^nu)E(t^NR[[t]])X_\alpha(t^nu)^{-1}\subseteq E(t^MR[[t]]).
\end{equation}
Clearly, this statement and~\eqref{eq:betaN} together imply~\eqref{eq:[tt]-v}.

By~\cite[Lemma 11]{PS} we know that if $N\ge 3$, then $E(t^NR[[t]])$ is contained in the subgroup
of $E(R[[t]])$ generated by $X_\gamma(V_\gamma\otimes_R t^{\lfloor\frac N3\rfloor}R[[t]])$
for all $\gamma\in\Phi_P\setminus\ZZ\alpha$. On the other hand, for any such $\gamma$ by the Chevalley commutator
formula~\eqref{eq:Chev} of Lemma~\ref{lem:rootels} we have
$$
\left[X_\alpha(t^nu),X_\gamma(V_\gamma\otimes_R t^{\lfloor\frac N3\rfloor}R[[t]])\right]\subseteq
E\l(t^{\lfloor\frac N3\rfloor-|\Phi_P|\cdot |n|}R[[t]]\r).
$$
This implies the claim~\eqref{eq:nNM}.
\end{proof}

%\begin{rem}
%Under the assumptions of Lemma~\ref{lem:field}, if $G$ is a simply connected group,
%then
%$$
%K_1^{G,P}(k)=K_1^{G,P}\bigl(k(t)\bigr)=K_1^{G,P}\bigl(k((t))\bigr)
%$$
%by~\cite[Th\'eor\`eme 5.8]{Gil}. In general this is not true. For example, if $G=\PGL_n$, then the short exact sequence
%$1\to\mu_n\to\SL_n\to\PGL_n\to 1$ readily implies
%$
%K_1^{\PGL_n,P}(l)=l^\times/l^{\times n}
%$
%for any field $l$. The map  $K_1^{\PGL_n,P}\bigl(k(t)\bigr)\to K_1^{\PGL_n,P}\bigl(k((t))\bigr)$ is
%not injective.
%\end{rem}

\section{Proof of the main results}

\subsection{Diagonal argument for loop reductive groups}

Our main results are based on the following observation.

\begin{lem}["diagonal argument"]\label{lem:diag}
Let $k$ be a field of characteristic $0$. Let $G$ be a loop reductive group over $R=k[x_1^{\pm 1},\ldots,x_n^{\pm 1}]$.
For any integer $d>0$, denote by $f_{z,d}$ (respectively, $f_{w,d}$) the composition of $k$-homomorphisms
$$
R\to k[z_1^{\pm 1},\ldots,z_n^{\pm 1},w_1^{\pm 1},\ldots,w_n^{\pm 1}]
\to k[z_1^{\pm 1},\ldots,z_n^{\pm 1},(z_1w_1^{-1})^{\pm \frac1d},\ldots,(z_nw_n^{-1})^{\pm \frac1d}]
$$
sending $x_i$ to $z_i$ (respectively, to $w_i$) for any $1\le i\le n$. Then there is $d>0$ such that
$$
f_{z,d}^*(G)\cong f_{w,d}^*(G)
$$
as group schemes over $k[z_1^{\pm 1},\ldots,z_n^{\pm 1},(z_1w_1^{-1})^{\pm \frac1d},\ldots,(z_nw_n^{-1})^{\pm \frac1d}]$.
\end{lem}
\begin{proof}
Let $G_0$ be a split reductive group over $k$ such that $G$ is a twisted form of $G_0$. Let $A_0=\Aut(G_0)$
be the group scheme of automorphisms of $G_0$. Denote by $\bar k$ the algebraic closure of $k$,
and by $\Gamma$ the Galois group $\Gal(\bar k/k)$.
We also introduce the following auxiliary notation: we write $X_x$ for the $k$-scheme
$\Spec k[x_1^{\pm 1},\ldots,x_n^{\pm 1}]$,
$X_z$ for $\Spec k[z_1^{\pm 1},\ldots,z_n^{\pm 1}]$, etc.

According to Definition~\ref{def:loop}, $G$ is given by a cocycle $\eta$ in $H^1\bigl(\pi_0(X_x,e),A_0(\bar k)\bigr)$.
Considering the description~\eqref{eq:pi1X} of~\S~\ref{ssec:loop}, we can assume that
$$
\eta\in H^1\Bigl(\Gal\bigl(\bar k[x_1^{\pm\frac 1d},\ldots,x_n^{\pm\frac 1d}]/k[x_1^{\pm 1},\ldots,x_n^{\pm 1}]\bigr),
A_0(\bar k)\Bigr)
$$
for some integer $d>0$, and we know that
$$
\Gal\bigl(\bar k[x_1^{\pm\frac 1d},\ldots,x_n^{\pm\frac 1d}]/k[x_1^{\pm 1},\ldots,x_n^{\pm 1}]\bigr)=M_x\semil\Gamma,
$$
where $M_x\cong (\ZZ/d\ZZ)^n$ acts on $\bar k[x_1^{\pm\frac 1d},\ldots,x_n^{\pm\frac 1d}]$ by sending
$x_i^{\frac 1d}$ to $\xi_d^{k_i}x_i^{\frac 1d}$, for any $(k_1,\ldots,k_n)\in M_x$.
We will denote by $M_z$ and $M_w$ respectively the group $(\ZZ/d\ZZ)^n$ operating in the same way on
$\bar k[z_1^{\pm\frac 1d},\ldots,z_n^{\pm\frac 1d}]$ and $\bar k[w_1^{\pm\frac 1d},\ldots,w_n^{\pm\frac 1d}]$.

Denote by $i_z$ (respectively, $i_w$) the $k$-homomorphism
$$
k[x_1^{\pm 1},\ldots,x_n^{\pm 1}]\to k[z_1^{\pm 1},\ldots,z_n^{\pm 1},w_1^{\pm 1},\ldots,w_n^{\pm 1}]
$$
sending $x_i$ to $z_i$ (respectively, to $w_i$) for any $1\le i\le n$. Consider the images
$i_z^*(\eta)$ and $i_w^*(\eta)$ of $\eta$ in $H^1\bigl(\pi_0(X_z\times_k X_w,e),A_0(\bar k)\bigr)$
as elements of
$$
H^1\bigl((M_z\times M_w)\semil\Gamma,A_0(\bar k)\bigr).
$$
%where
%$$
%(M_z\times M_w)\semil\Gamma=\Gal\bigl(\bar k[z_1^{\pm\frac 1d},\ldots,z_n^{\pm\frac 1d},
%w_1^{\pm\frac 1d},\ldots,w_n^{\pm\frac 1d}]/k[z_1^{\pm 1},\ldots,z_n^{\pm 1},w_1^{\pm 1},\ldots,w_n^{\pm 1}]\bigr).
%$$
Denote by $\Delta$ the diagonal subgroup of $M_z\times M_w$. The subgroup $\Delta\semil\Gamma$
of $(M_z\times M_w)\semil\Gamma$ is closed, and it is straightforward to check
that
$$
i_z^*(\eta)|_{\Delta\semil\Gamma}=i_w^*(\eta)|_{\Delta\semil\Gamma}.
$$
Since
$$
\bigl(\bar k[z_1^{\pm\frac 1d},\ldots,z_n^{\pm\frac 1d},
w_1^{\pm\frac 1d},\ldots,w_n^{\pm\frac 1d}]\bigr)^{\Delta\semil\Gamma}=
k[z_1^{\pm 1},\ldots,z_n^{\pm 1},
(z_1w_1^{-1})^{\pm\frac 1d},\ldots,(z_nw_n^{-1})^{\pm\frac 1d}],
$$
we conclude that $f_{z,d}^*(G)\cong f_{w,d}^*(G)$, as required.

\end{proof}

We introduce additional notation that will be used every time when we apply Lemma~\ref{lem:diag} in
proofs of other statements.

\begin{nota}\label{nota:diag}
In the setting of the claim of Lemma~\ref{lem:diag}, set
$$
t_i=(z_iw_i^{-1})^{1/d},\quad 1\le i\le n,
$$
where $z_i,w_i$, and $d$ are as in that lemma. Note that this is equivalent to
$$
z_i=w_it_i^d,\quad 1\le i\le n.
$$
We denote by $G_z$ the group scheme over
$k[z_1^{\pm 1},\ldots,z_n^{\pm 1}]$ which is the pull-back of $G$ under the $k$-isomorphism
$$k[x_1^{\pm 1},\ldots,x_n^{\pm 1}]\xrightarrow{x_i\mapsto z_i}k[z_1^{\pm 1},\ldots,z_n^{\pm 1}].$$
The group scheme $G_w$ over $k[w_1^{\pm 1},\ldots,w_n^{\pm 1}]$ is defined analogously. Note that $G_z$ and $G_w$ are isomorphic after pull-back to
$$
k[z_1^{\pm 1},\ldots,z_n^{\pm 1},t_1^{\pm 1},\ldots,t_n^{\pm 1}]=
k[w_1^{\pm 1},\ldots,w_n^{\pm 1},t_1^{\pm 1},\ldots,t_n^{\pm 1}].
$$
\end{nota}

\subsection{Proof of Theorem~\ref{thm:2fields} on ${\mathcal R}$-equivalence class groups}\label{ssec:2fields}

%To illustrate the use of Lemma~\ref{lem:diag} in the proof of Theorem~\ref{thm:main} below, we first prove
%the much simpler Theorem~\ref{thm:2fields} using the same technique.

\begin{proof}[Proof of Theorem~\ref{thm:2fields}]
The surjectivity of the natural map
$$
G\bigl(k(x_1,\ldots,x_n)\bigr)/{\mathcal R}\to G\bigl(k((x_1))\ldots((x_n))\bigr)/{\mathcal R}
$$
follows from Corollary~\ref{cor:2fields-surj}. To prove the injectivity,
recall that, since $G$ has a maximal torus over
$k[x_1^{\pm 1},\ldots,x_n^{\pm 1}]$,  it is loop reductive by~\cite[Corollary 6.3]{GiPi-mem}.
Thus, we can apply Lemma~\ref{lem:diag} to $G$. We use Notation~\ref{nota:diag}.
%We also set
%$$
%G_z=f_{z,d}^*(G),\qquad G_w=f_{w,d}^*(G).
%$$
%We set
%$w_i=z_it_i^{-d_i}$, $1\le i\le n$, so that $z_i=w_it^{d_i}$ and $(p_2\circ\rho)^\sharp(x_i)=w_i$.

%\begin{tikzpicture}
%\matrix (m) [matrix of math nodes,row sep=35pt,column sep=3em,minimum width=2em]
%  {
%     G\Bigl(k(x_1,\ldots,x_n)\Bigr)/{\mathcal R} & G\Bigl(k((x_1))\ldots((x_n))\Bigr)/{\mathcal R} \\
%    \begin{array}{c}G_z\Bigl(k(z_1,\ldots,z_n,t_1,\ldots,t_n)\Bigr)/{\mathcal R}\\
%                     =\strut G_w\Bigl(k(w_1,\ldots,w_n,t_1,\ldots,t_n)\Bigr)/{\mathcal R}
%     \end{array}
%                     &   G_w\Bigr(k(w_1,\ldots,w_n)((t_1))\ldots((t_n))\Bigr)/{\mathcal R} \\
%  };
%\path[-stealth]
%     (m-1-1)       edge node [above] {$j$}  (m-1-2)
% %%   (m-2-1) edge  [out=10,in=-10] node [right] {$\st t_i\mapsto 1,\ z_i\mapsto x_i$} (m-1-1.south east)
%%    ([xshift=-60pt]m-2-1.north) edge  [bend left] node [left] {$\st s^\sharp:\,t_i\mapsto 1,\ z_i\mapsto x_i$}
%%        ([xshift=-60pt]m-1-1.south)
%    (m-2-1.east|-m-2-2) edge node [above] {$\tilde\jmath$} (m-2-2)
%   % (m-2-2.north) edge node [right] {$\st \sigma\colon t_i\mapsto 1,\ w_i\mapsto x_i$} (m-1-2.south)
%;
%\path[-stealth](m-1-2) edge node [left] {$\st \lambda\colon x_i\mapsto z_i=w_it_i^{d}$}  (m-2-2);
%\path [-stealth]
%    ([xshift=10pt]m-1-1.south) edge node [right] {$\st f_{w,d}\colon x_i\mapsto w_i=z_it_i^{-d}$} ([xshift=10pt]m-2-1.north)
%    (m-1-1)  edge node [left] {$\st f_{z,d}\colon x_i\mapsto z_i$}(m-2-1)
%;
%%\path[-]  (m-3-1) edge [double distance=2pt] (m-2-1);

%\end{tikzpicture}

Consider the following commutative diagram, where the horizontal maps $j_1$ and $j_2$ are the natural ones.

\begin{tikzpicture}
\matrix (m) [matrix of math nodes,row sep=25pt,column sep=3em,minimum width=2em]
  {
     G\Bigl(k(x_1,\ldots,x_n)\Bigr)/{\mathcal R} & G\Bigl(k((x_1))\ldots((x_n))\Bigr)/{\mathcal R} \\
    G_z\Bigl(k(z_1,\ldots,z_n,t_1,\ldots,t_n)\Bigr)/{\mathcal R}& G_z\Bigl(k((z_1))\ldots((z_n))\Bigr)/{\mathcal R}   \\
    G_w\Bigl(k(w_1,\ldots,w_n,t_1,\ldots,t_n)\Bigr)/{\mathcal R} & G_w\Bigr(k(w_1,\ldots,w_n)((t_1))\ldots((t_n))\Bigr)/{\mathcal R}\\
  };
\path[-stealth]
     (m-1-1)       edge node [above] {$j_1$}  (m-1-2)
 %%   (m-2-1) edge  [out=10,in=-10] node [right] {$\st t_i\mapsto 1,\ z_i\mapsto x_i$} (m-1-1.south east)
%    ([xshift=-60pt]m-2-1.north) edge  [bend left] node [left] {$\st s^\sharp:\,t_i\mapsto 1,\ z_i\mapsto x_i$}
%        ([xshift=-60pt]m-1-1.south)
    (m-3-1.east) edge node [above] {$j_2$} (m-3-2)
    (m-3-1.east) edge node [below] {$\cong$} (m-3-2)
   % (m-2-2.north) edge node [right] {$\st \sigma\colon t_i\mapsto 1,\ w_i\mapsto x_i$} (m-1-2.south)
;
\path[-stealth](m-1-2) edge node [left] {$\st f_2\colon x_i\mapsto z_i$}  (m-2-2);
\path[-stealth](m-1-2) edge node [right] {$\cong$}  (m-2-2);
\path [-stealth]
    %([xshift=10pt]m-1-1.south) edge node [right] {$\st f_{w,d}\colon x_i\mapsto w_i=z_it_i^{-d}$} ([xshift=10pt]m-2-1.north)
    (m-1-1)  edge node [left] {$\st f_1\colon x_i\mapsto z_i$}(m-2-1)
    (m-1-1)  edge node [right] {$\cong$}(m-2-1)
    (m-2-1) edge node [right] {$\cong$} (m-3-1)
    (m-2-1) edge node [left] {$\st g_1\colon z_i\mapsto w_it_i^d$} (m-3-1)
    (m-2-2) edge node [left] {$\st g_2\colon  z_i\mapsto w_it_i^d$} (m-3-2)
;
%\path[-]  (m-3-1) edge [double distance=2pt] (m-2-1);

\end{tikzpicture}

The map $f_1$ in this diagram is an isomorphism, since $G_z$ is defined over $k(z_1,\ldots,z_n)$, and by~\cite[\S 16.2, Proposition 2]{Vos},
for any reductive group $H$ over an infinite field $l$ one has
$H(l)/{\mathcal R}\cong H\bigl(l(t)\bigr)/{\mathcal R}$. The map $f_2$ is an isomorphism be definition.
The map $g_1$ is an isomorphism by Lemma~\ref{lem:diag}.
The map $j_2$ is an isomorphism, since $G_w$ is defined over $k(w_1,\ldots,w_n)$, and by~\cite[Corollaire 0.3]{Gil-spec} for any reductive group $H$ over a
field $l$ of characteristic $\neq 2$ one has $H(l)/{\mathcal R}\cong H\bigl(l((t))\bigr)/{\mathcal R}$.

%Consider $G$ as a group scheme over the field $k(x_1,\ldots,x_n)$ and
%$\tilde G=f_{z,d}^*(G)$ as a group scheme defined over the field $k(z_1,\ldots,z_n)$. Then
%the map
%$$
%f_{z,d}:G\bigl(k(x_1,\ldots,x_n)\bigr)/{\mathcal R}\xrightarrow{x_i\mapsto z_i} {\tilde G}\bigl(k(z_1,\ldots,z_n,t_1,\ldots,t_n)\bigr)/{\mathcal R}
%$$
%is an isomorphism. Similarly, considering $\tilde G=f_{z,w}^*(G)$ as a group scheme defined over the
%infinite field $k(w_1,\ldots,w_n)$, we conclude that the map $\tilde\jmath$ is an isomorphism.
Since
$$
g_2\circ f_2\circ j_1=j_2\circ g_1\circ f_1
$$
is an isomorphism, we conslude that the map $j_1$ is injective.
\end{proof}

%\begin{rem}
%Note that Theorem~\ref{thm:2fields} is not true if we omit the assumption that $G$ is simply connected.
%For example, if $G=\PGL_n$, $n\ge 3$, then the short exact sequence
%$$1\to\mu_n\to\SL_n\to\PGL_n\to 1$$ readily implies
%$
%K_1^{\PGL_n}(l)=l^\times/l^{\times n}
%$
%for any field $l$. The map  $K_1^{\PGL_n}\bigl(k(t)\bigr)\to K_1^{\PGL_n}\bigl(k((t))\bigr)$ is
%not injective.
%\end{rem}

\subsection{Proof of Theorem~\ref{thm:main}}\label{ssec:mainproof}

%In this section we .
%\begin{thm}\label{thm:main}
%Let $k$ be an arbitrary field, and let $G$ be a loop reductive group over $X=\Spec k[x_1^{\pm 1},\ldots,x_n^{\pm 1}]$,
%such that every semisimple normal subgroup of $G$ contains $(\Gm_{,X})^2$.
%%Assume that there
%%exists a reductive group $G_0$ over $k$, and a $k$-morphism $\pi:X\to X$ which is a finite \'etale Galois cover, such that
%%$(G_0)_X\cong \pi^*(G)$ as $X$-group schemes.
%Then the natural map
%$$
%K_1^G\bigl(k[x_1^{\pm 1},\ldots,x_n^{\pm 1}]\bigr)\to K_1^G\bigl(k((x_1))\ldots((x_n))\bigr)
%$$
%is injective.
%\end{thm}
In order to prove Theorem~\ref{thm:main}, we still need to prove some technical lemmas.

\begin{lem}\label{lem:zwt}
Let $k$ be an arbitrary field, $A$ be a commutative $k$-algebra, and let $G$ be a reductive
group defined over
$A[z_1^{\pm 1},\ldots,z_n^{\pm 1}]$ such
that every semisimple normal subgroup of $G$  contains $(\Gm_{,k})^2$.
%Let $A=k[(\Gm_{,k})^s\times_k(\Ga_{,k})^r]$ for some $s,r\ge 0$.
For any set of integers $d_i>0$, $1\le i\le n$, the map
$$
K_1^G\bigl(A[z_1^{\pm 1},\ldots,z_n^{\pm 1},t_1,\ldots,t_n]\bigr)\xrightarrow{z_i\mapsto w_it_i^{d_i}}
K_1^G\bigl(A\otimes _k k(w_1,\ldots,w_n)[t_1^{\pm 1},\ldots,t_n^{\pm 1}]\bigr)
$$
is injective.
\end{lem}
\begin{proof}
We prove the claim by induction on $n\ge 0$. The case $n=0$ is trivial.
To prove the induction step for $n\ge 1$, it is enough to show that
$$
\phi:K_1^G\bigl(A\otimes_k k[z_1^{\pm 1},\ldots,z_n^{\pm 1},t_1,\ldots,t_n]\bigr)
\xrightarrow{z_1\mapsto w_1t_1^{d_1}}
K_1^G\bigl(A\otimes_k k(w_1)[t_1^{\pm 1}][z_{2}^{\pm 1},\ldots,z_n^{\pm 1},t_{2},\ldots,t_n]\bigr)
$$
is injective. Indeed, after that we can apply the induction assumption with $k$ substituted by $k(w_1)$ and $A$
substituted by $A\otimes_k k(w_1)[t_1^{\pm 1}]$. Set
$$
B=A[z_{2}^{\pm 1},\ldots,z_n^{\pm 1},t_{2},\ldots,t_n]$$
and omit for simplicity the subscript $1$. Then we need to show that the map
$$
\phi:K_1^G\bigl(B[z^{\pm 1},t]\bigr)\xrightarrow{z\mapsto wt^d} K_1^G\bigl(B\otimes_k k(w)[t^{\pm 1}]\bigr)
$$
is injective. Here $G$ is defined over $B[z^{\pm 1}]$.
%We factor the map $\phi$ as
%$$
%K_1^G(B[z^{\pm 1},t])\to K_1^G(B[z^{\pm 1},t^{\pm 1}])\xrightarrow{z\mapsto wt^d} K_1^G(B\otimes_k k(w)[t^{\pm 1}]).
%$$
%The first map here is injective by ..., so it is enough to prove that the second map is injective.
We have
$$
B\otimes_k k(w)[t^{\pm 1}]=
\varinjlim\limits_{g} B\otimes_k k[w^{\pm 1}]_g[t^{\pm 1}]=\varinjlim\limits_{g} B\otimes_k k[w^{\pm 1},t^{\pm 1}]_g,
$$
where $g=g(w)$ runs over all monic polynomials in $k[w]$ with $g(0)\neq 0$. Since $\phi(z)=wt^d$, we have
$g(w)=g(\phi(z)t^{-d})=t^{-Nd}f(t)$ for a suitable integer $N$, where $f(t)$ is a polynomial in $t$ with coefficients in
$k[\phi(z)^{\pm 1}]$ such that its leading coefficient is invertible. Then by Lemma~\ref{lem:f} the natural map
$$
K_1^G\bigl(B[z^{\pm 1},t]\bigr)\xrightarrow{z\mapsto wt^d} K_1^G\bigl(B\otimes_k k[w^{\pm 1},t^{\pm 1}]_g\bigr)=
K_1^G\bigl(B\otimes_k k[\phi(z)^{\pm 1},t]_{tf}\bigr)
$$
is injective. Since $K_1^G$ commutes with filtered direct limits, we conclude that $\phi$ is injective.
\end{proof}

\begin{lem}\label{lem:const-inj}
Let $k$ be an arbitrary field, let $A$ be a commutative $k$-algebra, and let $G$
be a reductive group scheme over $A$ such that every
semisimple normal subgroup of $G$ contains $(\Gm_{,A})^2$. For any $n\ge 0$
the natural map
$$
K_1^G\bigl(A[t_1^{\pm 1},\ldots,t_n^{\pm 1}]\bigr)\to K_1^G\bigl(A\otimes_k k(t_1,\ldots,t_n)\bigr)
$$
is injective.
\end{lem}
\begin{proof}
We prove the claim by induction on $n$; the case $n=0$ is trivial.
Set $l=k(t_1,\ldots,t_{n-1})$. By the inductive hypothesis, the map
$$
K_1^G\bigl(A[t_1^{\pm 1},\ldots,t_n^{\pm 1}]\bigr)\to K_1^G\bigl(A[t_n^{\pm 1}]\otimes_k l\bigr)=
K_1^G\bigl(A\otimes_k l[t_n^{\pm 1}]\bigr)
$$
is injective, so it remains to prove the injectivity of the map
$$
K_1^G\bigl(A\otimes_k l[t_n^{\pm 1}]\bigr)\to K_1^G\bigl(A\otimes_k l(t_n)\bigr).
$$
We have $l(t_n)=\varinjlim\limits_{g} l[t_n]_{t_ng}$, where $g\in l[t_n]$ runs over all monic polynomials
coprime to $t_n$. Since $K_1^G$ commutes with filtered direct limits, it remains to show that every map
\begin{equation}\label{eq:coinj}
K_1^G(A\otimes_k l[t_n^{\pm 1}])\to K_1^G(A\otimes_k l[t_n]_{t_ng})
\end{equation}
is injective.
%The natural map
%$$
%G(A\otimes_k l[t_n^{\pm 1}])\to G(A\otimes_k l[t_n]_f)
%$$
%is injective, and we identify the domain with a subgroup of the target.
Assume that
$$
x\in G(A\otimes_k l[t_n^{\pm 1}])\cap E(A\otimes_k l[t_n]_{t_ng}).
$$
By~\cite[Lemma 2.3]{St-poly}
there exist $x_1\in E(A\otimes_k l[t_n]_{t_n})$ and $x_2\in E(A\otimes_k l[t_n]_g)$
such that $x=x_1x_2$. We have $x,x_1\in G(A\otimes_k l[t_n^{\pm 1}])$, therefore,
$x_2\in G(A\otimes_k l[t_n^{\pm 1}])$.
Since
$$
G(A\otimes_k l[t_n^{\pm 1}])\cap G(A\otimes_k l[t_n]_g)=G(A\otimes_k l[t_n]),
$$
we have $x_2\in G(A\otimes_k l[t_n])\cap E(A\otimes_k l[t_n]_g)$.
By Lemma~\ref{lem:f} this implies that $x_2\in E(A\otimes_k l[t_n])$. Summing up, we have
$x=x_1x_2\in E(A\otimes_k l[t_n^{\pm 1}])$.
Therefore, the map~\eqref{eq:coinj}
is injective.
%The proof is the same as for the case $A=k$, which was established in the
%proof of~\cite[Corollary 6.2]{St-poly}. It follows by induction from Lemma~\ref{lem:f} and the
%following statement~\cite[Lemma 2.3]{St-poly}. Let
%$B$ be any commutative ring, $H$ be a reductive group over $B$ such that every
%semisimple normal subgroup of $H$ is isotropic over $B$ and contains $(\Gm_{,B})^2$.
%Let $f,g\in B$ be such that $fB+gB=B$.
%Then the sequence of pointed sets
%$$
%K_1^H(B)\xrightarrow{\st g\mapsto (g,g)} K_1^H(B_f)\times K_1^H(B_g)
%\xrightarrow{\st (g_1,g_2)\mapsto g_1{g_2}^{-1}} K_1^H\bigl(B_{fg}\bigr)
%$$
%is exact.
\end{proof}

\begin{lem}\label{lem:Lau-frac}
Let $k$ be a  field of characteristic $0$, and let $G$ be a reductive group over
$X=\Spec k[x_1^{\pm 1},\ldots,x_n^{\pm 1}]$, having a maximal $X$-torus and
such that every semisimple normal subgroup of $G$ contains $(\Gm_{,X})^2$. Then

%Let $k_0$ be an algebraically closed field of characteristic $0$, $X=(\Gm_{,k_0})^n=\Spec k_0[x_1^{\pm 1},\ldots,x_n^{\pm 1}]$,
%and let $k$ be a field extension of $k_0$. Let $G$ be a reductive group over $X$.

%Assume that every semisimple normal subgroup of $G$ is isotropic over $X_k=\Spec k[x_1^{\pm 1},\ldots,x_n^{\pm 1}]$
%and contains $(\Gm_{,X_k})^2$.

(i) the natural map
% спец. типы G расщепимые  чисто полиномиально
$$
K_1^G\bigl(k[x_1^{\pm 1},\ldots,x_n^{\pm 1}]\bigr)\to K_1^G\bigl(k(x_1,\ldots,x_n)\bigr)
$$
is injective;

(ii) one has $K_1^G\bigl(k[x_1^{\pm 1},\ldots,x_n^{\pm 1}]\bigr)=
K_1^G\bigl(k[x_1^{\pm 1},\ldots,x_n^{\pm 1},y_1,\ldots,y_m]\bigr)$
for any $m\ge 0$.
\end{lem}
\begin{proof}
First we show that for any $m\ge 0$ the natural map
$$
K_1^G\bigl(k[x_1^{\pm 1},\ldots,x_n^{\pm 1},y_1,\ldots,y_m]\bigr)\to K_1^G\bigl(k(x_1,\ldots,x_n)[y_1,\ldots,y_m]\bigr)
$$
is injective. This includes (i). For shortness, we write ${\bf y}$ instead of $y_1,\ldots,y_m$.

As in~Theorem~\ref{thm:2fields}, we note that $G$ is loop reductive over $k[x_1^{\pm 1},\ldots,x_n^{\pm 1}]$ by~\cite[Corollary 6.3]{GiPi-mem}.
We apply Lemma~\ref{lem:diag} to $G$, and we use Notation~\ref{nota:diag}.
Consider the following commutative diagram. In this diagram, the horizontal maps $j_1$ and $j_2$ are the natural ones,
and all maps always take variables $t_i$ to $t_i$, $1\le i\le n$, and $\mathbf{y}$ to $\mathbf y$.
The isomorphisms $g_1$ and $g_2$ exist by Lemma~\ref{lem:diag}.

\begin{tikzpicture}
\matrix (m) [matrix of math nodes,row sep=25pt,column sep=3em,minimum width=2em]
  {
   K_1^G\Bigl(k[x_1^{\pm 1},\ldots,x_n^{\pm 1},{\bf y}]\Bigr) & K_1^G\Bigl(k(x_1,\ldots,x_n)[{\bf y}]\Bigr) \\
    K_1^{G_z}\Bigl(k[z_1^{\pm 1},\ldots,z_n^{\pm 1},t_1,\ldots,t_n,{\bf y}]\Bigr)& K_1^{G_z}\Bigl(k(z_1,\ldots,z_n,t_1,\ldots,t_n)[{\bf y}]\Bigr)   \\
    K_1^{G_z}\Bigl(k(w_1,\ldots,w_n)[t_1^{\pm 1},\ldots,t_n^{\pm 1},{\bf y}]\Bigr) &  \\
    K_1^{G_w}\Bigl(k(w_1,\ldots,w_n)[t_1^{\pm 1},\ldots,t_n^{\pm 1},{\bf y}]\Bigr) & K_1^{G_w}\Bigl(k(w_1,\ldots,w_n,t_1,\ldots,t_n)[{\bf y}]\Bigr)\\
  };
\path[-stealth]
     (m-1-1)       edge node [above] {$j_1$}  (m-1-2)
    (m-4-1.east) edge node [above] {$j_2$} (m-4-2)
%    (m-3-1.east) edge node [below] {$\cong$} (m-3-2)
;
\path[-stealth](m-1-2) edge node [left] {$\st f_2\colon x_i\mapsto z_i$}  (m-2-2);
\path [-stealth]
    %([xshift=10pt]m-1-1.south) edge node [right] {$\st f_{w,d}\colon x_i\mapsto w_i=z_it_i^{-d}$} ([xshift=10pt]m-2-1.north)
    (m-1-1)  edge node [left] {$\st f_1\colon x_i\mapsto z_i$}(m-2-1)
%    (m-1-1)  edge node [right] {$\cong$}(m-2-1)
    (m-2-1) edge node [left] {$\st h\colon z_i\mapsto w_it_i^d$} (m-3-1)
    (m-3-1) edge node [left] {$\st g_1$} (m-4-1)
    (m-3-1) edge node [right] {$\cong$} (m-4-1)
    (m-2-2) edge node [left] {$\st g_2\colon  z_i\mapsto w_it_i^d$} (m-4-2)
    (m-2-2) edge node [right] {$\cong$}  (m-4-2);
;

\end{tikzpicture}

In order to prove that $j_1$ is injective, it is enough to show that all maps $j_2,g_1,h,f_1$ are injective.
The map $j_2$ is injective by Lemma~\ref{lem:const-inj}. As explained above, $g_1$ is an isomorphism.
The map $h$ is injective by Lemma~\ref{lem:zwt}. Finally, the map $f_1$ is injective, since it has a retraction
that sends $z_i$ to $x_i$ and $t_i$ to $0$. Therefore, the map $j_1$ is injective.

Now we prove (ii). Consider the commutative diagram
\begin{equation*}
\xymatrix{
K_1^G\Bigl(k[x_1^{\pm 1},\ldots,x_n^{\pm 1}][{\bf y}]\Bigr)\ar[rr]^{y_i\mapsto 0}\ar[d]_{} &&
K_1^G\Bigl(k[x_1^{\pm 1},\ldots,x_n^{\pm 1}]\Bigr)\ar[d]^{}  &\\
K_1^G\Bigl(k(x_1,\ldots,x_n)[{\bf y}]\Bigr)\ar[rr]^{y_i\mapsto 0}&& K_1^G\Bigl(k(x_1,\ldots,x_n)\Bigr).
}
\end{equation*}
The bottom arrow is an isomorphism by~\cite[Theorem 1.2]{St-poly}. The vertical arrows are injective
by the previous paragraph Therefore, the top arrow
$$
K_1^G\Bigl(k[x_1^{\pm 1},\ldots,x_n^{\pm 1}][{\bf y}]\Bigr)\xrightarrow{y_i\mapsto 0}
K_1^G\Bigl(k[x_1^{\pm 1},\ldots,x_n^{\pm 1}]\Bigr)
$$
is also injective. Since it has a section, it is an isomorphism.

\end{proof}

\begin{proof}[Proof of Theorem~\ref{thm:main}.]
We prove the injectivity claim by induction on $n$ starting with the trivial case $n=0$. To prove the induction step, it is enough
to show that the map
$$
j:K_1^G\Bigl(k[x_1^{\pm 1},\ldots,x_n^{\pm 1}]\Bigr) \to K_1^G\Bigl(k((x_1))[x_2^{\pm 1},\ldots,x_n^{\pm 1}]\Bigr)
$$
is injective. The latter follows from the injectivity of the composition
$$
j_1\colon K_1^G\Bigl(k[x_1^{\pm 1},\ldots,x_n^{\pm 1}]\Bigr) \xrightarrow{j} K_1^G\Bigl(k((x_1))[x_2^{\pm 1},\ldots,x_n^{\pm 1}]\Bigr)\to K_1^G\Bigl(k[x_2^{\pm 1},\ldots,x_n^{\pm 1}]((x_1))\Bigr),
$$
which we proceed to establish.

The group $G$ is loop reductive by~\cite[Corollary 6.3]{GiPi-mem}, since it has a maximal torus. We apply Lemma~\ref{lem:diag} to $G$,
and we use Notation~\ref{nota:diag}.
Consider the following commutative diagram. Here $j_1,j_2$ are the natural maps, and
the isomorphism $g_1$ and the map $g_2$ exist by Lemma~\ref{lem:diag}.

\begin{tikzpicture}
\matrix (m) [matrix of math nodes,row sep=25pt,column sep=3em,minimum width=2em]
  {
     K_1^G\Bigl(k[x_1^{\pm 1},\ldots,x_n^{\pm 1}]\Bigr) & K_1^G\Bigl(k[x_2^{\pm 1},\ldots,x_n^{\pm 1}]((x_1))\Bigr) \\
    K_1^{G_z}\Bigl(k[z_1^{\pm 1},\ldots,z_n^{\pm 1},t_1^{\pm 1},\ldots,t_n^{\pm 1}]\Bigr)& K_1^{G_z}\Bigl(k[z_2^{\pm 1},\ldots,z_n^{\pm 1}]((z_1))\Bigr)   \\
    K_1^{G_w}\Bigl(k[w_1^{\pm 1},\ldots,w_n^{\pm 1},t_1^{\pm 1},\ldots,t_n^{\pm 1}]\Bigr) & K_1^{G_w}\Bigr(k[w_1^{\pm 1},\ldots,w_n^{\pm 1},t_2^{\pm 1},\ldots,t_n^{\pm 1}]((t_1))\Bigr)\\
  };
\path[-stealth]
     (m-1-1)       edge node [above] {$j_1$}  (m-1-2)
    (m-3-1.east) edge node [above] {$j_2$} (m-3-2)
%    (m-3-1.east) edge node [below] {$\cong$} (m-3-2)
;
\path[-stealth](m-1-2) edge node [left] {$\st f_2\colon x_i\mapsto z_i$}  (m-2-2);
\path[-stealth](m-1-2) edge node [right] {$\cong$}  (m-2-2);
\path [-stealth]
    (m-1-1)  edge node [left] {$\st f_1\colon x_i\mapsto z_i$}(m-2-1)
    (m-2-1) edge node [right] {$\cong$} (m-3-1)
    (m-2-1) edge node [left] {$\st g_1\colon z_i\mapsto w_it_i^d$} (m-3-1)
    (m-2-2) edge node [left] {$\st g_2\colon  z_i\mapsto w_it_i^d$} (m-3-2)
;
%\path[-]  (m-3-1) edge [double distance=2pt] (m-2-1);

\end{tikzpicture}

%\begin{tikzpicture}
%\matrix (m) [matrix of math nodes,row sep=35pt,column sep=3em,minimum width=2em]
%  {
%     K_1^G\Bigl(k[x_1^{\pm 1},\ldots,x_n^{\pm 1}]\Bigr) & K_1^G\Bigl(k[x_2^{\pm 1},\ldots,x_n^{\pm 1}]((x_1))\Bigr) \\
%    \begin{array}{c} K_1^{\tilde G}\Bigl(k[z_1^{\pm 1},\ldots,z_n^{\pm 1},t_1^{\pm 1},\ldots,t_n^{\pm 1}]\Bigr)\\
%                     =\strut K_1^{\tilde G}\Bigl(k[w_1^{\pm 1},\ldots,w_n^{\pm 1},t_1^{\pm 1},\ldots,t_n^{\pm 1}]\Bigr)
%     \end{array}
%                     &
%              K_1^{\tilde G}\Bigr(k[w_1^{\pm 1},\ldots,w_n^{\pm 1},t_2^{\pm 1},\ldots,t_n^{\pm 1}]((t_1))\Bigr) \\
%      %\\&\\
%  };
%\path[-stealth]
%     (m-1-1)       edge node [above] {$j$}  (m-1-2)
%% %   (m-2-1) edge  [out=10,in=-10] node [right] {$\st t_i\mapsto 1,\ z_i\mapsto x_i$} (m-1-1.south east)
%%    ([xshift=-60pt]m-2-1.north) edge  [bend left] node [left] {$\st s^\sharp\colon t_i\mapsto 1,\ z_i\mapsto x_i$}
%%        ([xshift=-60pt]m-1-1.south)
%    (m-2-1.east|-m-2-2) edge (m-2-2)
%%    ([xshift=10pt]m-2-2.north) edge [bend right] node [right] {$\st t_i\mapsto x_i,\ w_i\mapsto x_i^{1-d_i}$} ([xshift=10pt]m-1-2.south)
%;
%\path[-stealth](m-1-2) edge node [left] {$\st x_i\mapsto z_i=w_it_i^{d_i}$}  (m-2-2);
%\path [-stealth]
%    ([xshift=10pt]m-1-1.south) edge node [right] {$\st f_{w,d}\colon x_i\mapsto w_i=z_it_i^{-d_i}$} ([xshift=10pt]m-2-1.north)
%    (m-1-1)  edge node [left] {$\st f_{z,d}\colon x_i\mapsto z_i$}(m-2-1)
%;
%%\path[-]  (m-3-1) edge [double distance=2pt] (m-2-1);

%\end{tikzpicture}

In order to show that $j_1$ is injective, it is enough to show that $f_1$ and $j_2$ are injective.
The map $f_1$ is injective, since it has a retraction
that sends $z_i$ to $x_i$ and $t_i$ to $1$.
Set
$$
A=k[w_1^{\pm 1},\ldots,w_n^{\pm 1},t_2^{\pm 1},\ldots,t_n^{\pm 1}].
$$
By Lemma~\ref{lem:Lau-frac} (ii) we have $K_1^{G_w}(A[t_1])=K_1^{G_w}(A)$, therefore, by
Corollary~\ref{cor:t-t-inj} the map $j_2$ is injective. Therefore, the map $j_1$ is injective.

To finish the proof of the theorem, it remains to note that, if $G$ is a semisimple group, the map
$$
K_1^G\bigl(k[x_1^{\pm 1},\ldots,x_n^{\pm 1}]\bigr)\to K_1^G\bigl(k((x_1))\ldots((x_n))\bigr)
$$
is surjective by Corollary~\ref{cor:dens-iso}.
\end{proof}

%\begin{cor}\label{cor:loop}
%Under the assumptions of Theorem~\ref{thm:Lau-((}, assume moreover that $G$ is a semisimple loop reductive
%group scheme. Then the natural map
%$$
%K_1^G\bigl(k[x_1^{\pm 1},\ldots,x_n^{\pm 1}]\bigr)\to K_1^G\bigl(k((x_1))\ldots((x_n))\bigr)
%$$
%is an isomorphism.
%\end{cor}
%\begin{proof}
%The surjectivity of the map holds by Theorem~\ref{thm:ChGP} of V. Chernousov, P. Gille
%and A. Pianzola. The injectivity is established in Theorem~\ref{thm:Lau-((}.
%\end{proof}

\section{Application to Lie tori}\label{sec:Lie}

Throughout this section, we assume that $k$ is an algebraically closed field of characteristic $0$.
We fix a compatible set of primitive $m$-th roots of unity $\xi_m\in k$, $m\ge 1$.

Let $G$ be an
adjoint simple algebraic group over $k$ (a Chevalley group), and $L=\Lie(G)$ the corresponding simple
Lie algebra over $k$. It is well-known that
$$
\Aut_k(L)\cong\Aut_k(G)\cong G\semir N,
$$
where $N$ is the finite group of automorphisms of the Dynkin diagram of the root system of $L$ and
$G$. Fix two integers $n\ge 0$, $m\ge 1$ and let
$$
\sigma=(\sigma_1,\ldots,\sigma_n)
$$
be an $n$-tuple of pairwise commuting elements of order $m$ in $\Aut_k(L)$. Such an $n$-tuple determines
a $\ZZ^n$-grading on $L$ with
$$
L_{i_1\ldots i_n}=\{x\in L\ |\ \sigma_j(x)=\xi_m^{i_j}x,\ 1\le j\le n\}.
$$

Set $R=k[x_1^{\pm 1},\ldots,x_n^{\pm 1}]$, and let $\tilde R=k[x_1^{\pm\frac 1m},\ldots,x_n^{\pm\frac 1m}]$, $m\ge 1$, be
another copy of $R$, considered as an $R$-algebra via the natural embedding $R\subseteq \tilde R$.
Then $\tilde R/R$ is a Galois ring extension with the Galois group
$$
\Gal(\tilde R/R)\cong (\ZZ/m\ZZ)^n.
$$

\begin{defn}
The \emph{multiloop Lie algebra $\La(L,\sigma)$} is the $\ZZ^n$-graded $k$-Lie subalgebra
$$
\La(L,\sigma)=\bigoplus\limits_{(i_1,\ldots,i_n)\in\ZZ^n}L_{i_1\ldots i_n}\otimes x_1^{\frac{i_1}m}\ldots
x_n^{\frac{i_n}m}
$$
of the $k$-Lie algebra $L\otimes_k \tilde R$.
\end{defn}

Note that, considered as an $R$-Lie algebra, the algebra $\La(L,\sigma)$ is an $\tilde R/R$-twisted form of the
$R$-Lie algebra $L\otimes_k R$. Indeed,
$$
\La(L,\sigma)\otimes_R\tilde R\cong (L\otimes_k R)\otimes_R \tilde R.
$$

Let $\Delta$ be a finite root system in the sense of~\cite{Bu} together with the $0$-vector, which
we include following the tradition in the theory of extended affine Lie algebras. We set
$\Delta^\times=\Delta\setminus\{0\}$, $Q=\ZZ\Delta$, and
$$
\Delta^\times_{ind}=\{\alpha\in\Delta^\times\ |\ {\textstyle\frac 12}\alpha\not\in\Delta\}.
$$

The importance of multiloop Lie algebras stems from the fact that they provide explicit realizations for
a class of infinite-dimensional Lie algebras over $k$ called Lie tori. This was shown by B. Allison,
S. Berman, J. Faulkner and A. Pianzola in~\cite{ABFP}.

\begin{defn}\cite[Def. 1.1.6]{ABFP}\label{defn:LT}
A \emph{Lie $\Lambda$-torus of type $\Delta$} is a $Q\times\Lambda$-graded Lie algebra
$\La=\bigoplus\limits_{(\alpha,\lambda)\in Q\times\Lambda}\La_\alpha^\lambda$ over $k$ satisfying
\begin{enumerate}
\item $\La_{\alpha}^\lambda=0$ for all $\alpha\in Q\setminus\Delta$ and all $\lambda\in\Lambda$.
\item $\La_{\alpha}^0\neq 0$ for all $\alpha\in\Delta^\times_{ind}$.
\item $\Lambda$ is generated by the set of all $\lambda\in\Lambda$ such that $\La_\alpha^\lambda\neq 0$
for some $\alpha\in\Delta$.
\item For all $(\alpha,\lambda)\in\Delta^\times\times\Lambda$ such that $\La_\alpha^\lambda\neq 0$, there exist elements
$e_{\alpha}^\lambda\in\La_\alpha^\lambda$ and $f_\alpha^\lambda\in\La_{-\alpha}^{-\lambda}$ satisfying
$$
\La_\alpha^\lambda=ke_{\alpha}^\lambda,\quad \La_{-\alpha}^{-\lambda}=kf_{\alpha}^\lambda,\quad\mbox{and}\quad
[[e_\alpha^\lambda,f_\alpha^\lambda],x]=\l<\beta,\alpha^\vee\r>x
$$
for all $x\in\La_\beta^\mu$, $(\beta,\mu)\in \Delta\times\Lambda$.
\item $\La$ is generated as a $k$-Lie algebra by the subspaces $\La_\alpha^\lambda$,
$(\alpha,\lambda)\in\Delta^\times\times\Lambda$.
\end{enumerate}
If $\Lambda=\ZZ^n$, then $n$ is called the \emph{nullity} of $\La$.
\end{defn}

In what follows we will always assume that
$$
\Lambda=\ZZ^n.
$$
By~\cite[Lemma 1.3.5 and Prop. 1.4.2]{ABFP}, if a centerless Lie torus $\La$ with $\Lambda\cong\ZZ^n$
is finitely generated over its centroid (fgc), then the centroid is isomorphic
as a $k$-algebra to
$$
k[\ZZ^n]\cong k[x_1^{\pm 1},\ldots,x_n^{\pm 1}]=R.
$$
Note that, according to an annouced result of E. Neher~\cite[Theorem 7(b)]{N}, all Lie tori are fgc,
except for just one class of Lie tori of type $A_n$ called quantum tori;
see~\cite[Remark 1.4.3]{ABFP}.

If a centerless Lie torus $\La$ is fgc, the Realization theorem~\cite[Theorem 3.3.1]{ABFP} asserts that $\La$ as a Lie
algebra over its centroid $R$ is $\ZZ^n$-graded isomorphic to a multiloop algebra
$\La(L,\sigma)$. In particular, the Lie torus $\La$ is a $\tilde R/R$-twisted form of a split simple Lie algebra
$L\otimes_k R$. Consequently, the group scheme of $R$-equivariant automorphisms $\Aut_R(\La)$ is a twisted form of $\Aut_R(L\otimes_k R)$, and
$\Aut_R(\La)^\circ$ is an adjoint simple reductive group over $R$. Moreover,
$$
\Lie(\Aut_R(\La)^\circ)\cong\La
$$
as Lie algebras over $R$, e.g.~\cite[Prop. 4.10]{GiPi07}.
% centroid(L(g,sigma))=R by GiPi07 Lemma 4.6

%Consider the following natural subgroup of $\Aut(\La)$, generated by the exponents of obviously nilpotent
%elements in $\La$:
%\begin{equation}
%\begin{array}{rl}
%%E_{exp}(\La)&=\l<\exp(u\cdot \ad_{e_\alpha^\lambda}),\ \exp(u\cdot\ad_{f_\alpha^\lambda}),\ u\in k,\ \alpha\in\Delta^\times,\ \lambda\in\Lambda\r>\\
%%&=\l<\exp(\ad_x),\ x\in\La^\lambda_\alpha,\ (\alpha,\lambda)\in \Delta^\times\times\Lambda\r>.
%E_{exp}(\La)=\l<\exp(\ad_x),\ x\in\La^\lambda_\alpha,\ (\alpha,\lambda)\in \Delta^\times\times\Lambda\r>.
%\end{array}
%\end{equation}

\begin{proof}[Proof of Theorem~\ref{thm:Lie}.]
First we show that the adjoint simple reductive group $G=\Aut_R(\La)^\circ$ over $R$ contains a closed $R$-subgroup
$S\cong (\Gm_{,R})^r$, where $r=\rank\Delta$.
Indeed, the Lie algebra $\La$ over $R$ is $Q$-graded, where $Q=\ZZ\Delta$. This grading naturally determines
a closed subgroup $S\cong (\Gm_{,R})^r$ of $\Aut_R(\La)$, where $r=\rank\Delta$. Namely, let $\Pi\subseteq\Delta$
be a system of simple roots, $|\Pi|=r$.
For any simple root $\alpha\in\Pi$, any commutative $R$-algebra $R'$, and any $c\in (R')^\times=\Gm(R')$,
there is a unique automorphism $t_{\alpha}(c)$ of $\La\otimes_R R'$ such that, for any $\lambda\in\ZZ^n$,
one has
$$
t_\alpha(c)(e_{\alpha}^\lambda)=ce_{\alpha}^\lambda,\qquad t_\alpha(c)(f_{\alpha}^\lambda)=c^{-1}f_{\alpha}^\lambda,
\qquad\mbox{and}
$$
$$
t_\alpha(c)(e_{\beta}^\lambda)=e_{\beta}^\lambda,\qquad t_\alpha(c)(f_{\beta}^\lambda)=f_{\beta}^\lambda
\qquad\mbox{for all}\ \beta\in\Pi,\ \beta\neq\alpha.
$$
Clearly, $S\subseteq\Aut_R(\La)^\circ$.

Conversely, the grading induced by the adjoint action of $S$ on $\Lie(\Aut_R(\La)^\circ)\cong\La$
is exactly the initial $Q$-grading. The system of simple roots $\Pi\subseteq\Delta$ determines a
decomposition $\Delta=\Delta^+\cup\Delta^-\cup\{0\}$, and by Lemma~\ref{lem:T-P} there exist
two opposite parabolic $R$-subgroups $P^+=U_{\Delta^+\cup\{0\}}$, $P^-=U_{\Delta^-\cup\{0\}}$ of
$G$, and their unipotent radicals are of the form $U_{\Delta^+}$ and $U_{\Delta^-}$ respectively.
Since $\Spec R$ is connected, the relative roots and relative roots subschemes with respect to
$P^{\pm}$ are defined over $\Spec R$. By Lemma~\ref{lem:rootels} (iv) the groups
$U_{\Delta^\pm}$ are generated by the root elements $X_\alpha(v)$, $\alpha\in\Delta^\pm$, $v\in\Lie(G)_\alpha$.
By Example~\ref{exa:exp} we can identify $X_\alpha(v)$ with $\exp(\ad_v)$. Therefore,
we have
$$
E_{P^+}(R)=\l<U_{\Delta^+}(R),U_{\Delta^-}(R)\r>=E_{exp}(\La).
$$

Since $\rank\Delta\ge 2$, the group $G$ contains $(\Gm_{,R})^2$. It also contains a maximal $R$-torus,
since by~\cite[p. 532]{GiPi08} the group $G$ is loop reductive.
It remains to apply Theorem~\ref{thm:main}.

%Since $\La$ is of the form $\La(L,\sigma)$, it is easy to see that $\La$ as a twisted form of $L\otimes_k R$ is given
%by a loop cocycle, e.g.~\cite[p. 532]{GiPi07}. Since the same cocycle determines
%$\Aut_R(\La)$ as a twisted form of $\Aut_R(L\otimes_k R)$, the group $\Aut_R(\La)^\circ$ is a loop reductive group.
%By Lemma~\ref{lem:ELa} , and we have $E_{exp}(\La)=E(R)$, the elementary subgroup
%of $G(R)$.

%If $\La$ is quasi-split, it is well-known that $G(K)=E(K)$,
%e.g. ~\cite[Th\'eor\`eme 6.1, Th\'eor\`eme 7.2]{Gil}.
\end{proof}

\renewcommand{\refname}{References}

\end{document}